\documentclass[hidelinks,onefignum,onetabnum]{siamart251216}

\usepackage{lipsum}
\usepackage{amsfonts}
\usepackage{graphicx}
\usepackage{epstopdf}
\usepackage{algorithmic}
\ifpdf
  \DeclareGraphicsExtensions{.eps,.pdf,.png,.jpg}
\else
  \DeclareGraphicsExtensions{.eps}
\fi
\usepackage{amsmath,amsfonts,amssymb,amscd,mathrsfs}
\usepackage{calc}
\usepackage{multicol}
\usepackage{empheq}
\usepackage{amsopn}

\usepackage{arydshln}
\usepackage[caption=false]{subfig}
\usepackage{enumitem}

\newcommand{\diff}{\mathrm{d}}

\definecolor{HsL}{RGB}{255,0,255}
\definecolor{HtL}{RGB}{90,13,67}

\newsiamremark{remark}{Remark}
\newsiamremark{hypothesis}{Hypothesis}
\crefname{hypothesis}{Hypothesis}{Hypotheses}
\newsiamthm{claim}{Claim}
\newsiamremark{fact}{Fact}
\crefname{fact}{Fact}{Facts}

\headers{Low-rank Framework for NMSSE}{Zhuohan Zhang, and Zhenning Cai}

\title{
A Low-Rank Hierarchical Framework for the Non-Markovian Stochastic Schrödinger Equation with Convergence Analysis
\thanks{
\funding{The work of Zhenning Cai was supported by the Academic Research Fund of the Ministry of Education of Singapore under Grant No. A-8002392-00-00.}
}
}

\author{
Zhuohan Zhang\thanks{
Department of Mathematics, National University of Singapore, 10 Lower Kent Ridge Road, Singapore 119076
(\email{zhuohan.zhang@u.nus.edu}).}
\and Zhenning Cai\thanks{ Corresponding author. 
Department of Mathematics, National University of Singapore, 10 Lower Kent Ridge Road, Singapore 119076  (\email{matcz@nus.edu.sg}).}
}

\ifpdf
\hypersetup{
  pdftitle={Low-rank Hierarchical Framework of NMSSE},
  pdfauthor={Z. Zhang and Z. Cai}
}
\fi

\newsiamremark{example}{Example}
\begin{document}

\maketitle

\begin{abstract}
We propose and analyze a novel numerical framework for the non-Markovian stochastic Schr\"{o}dinger equation (NMSSE) based on a low-rank approximation of the bath correlation functions. By decomposing the memory kernel into a finite-dimensional representation, we derive a truncated system of hierarchical equations that effectively balances computational tractability with physical fidelity. A rigorous convergence analysis is established for the hierarchical framework under mild assumptions. We demonstrate that our formulation serves as a mathematical generalization of the Hierarchy of Pure States (HOPS), encompassing it as a special case while offering a more flexible representation of non-Markovian effects. Numerical experiments across several benchmark models are presented to illustrate the validity and efficacy of the proposed method.
\end{abstract}



\section{Introduction}
\label{sec:intro}
Open quantum systems (OQS) play a central role in quantum physics and chemistry, with applications in quantum information processing \cite{Nielsen_Chuang_2010}, quantum optics \cite{carmichael1993OpenSystemsApproach}, and chemical physics \cite{Weiss2021QuantumDissipative}. A key difficulty in their simulation is the non-Markovian memory encoded in the bath correlation function (BCF). Depending on how the reduced dynamics is represented, existing numerical methods are often organized into two broad classes.

The first class evolves the reduced density matrix (RDM). Representative approaches include path integrals \cite{Feynman1948SpaceTimeApproach}, such as QuAPI \cite{Makri1992Improved} and the Dyson series iterative method \cite{cai2025SecondorderDiscretizationDyson}. To reduce the computational burden, i-QuAPI \cite{Makri1995NumericalPathIntegral, Makri1998QuantumDissipative} truncates memory, and \cite{zhan2025ReducingSpatialTemporal} combines low-rank approximation with the frozen Gaussian approximation. Another line of work starts from the Nakajima-Zwanzig equation \cite{Nakajima1958OnQuantumTheory, Zwanzig1960EnsembleMethod}, including TTM \cite{Cerrillo2014NonMarkovianDynamicalMaps} and HEOM \cite{Tanimura1990NonperturbativeExpansion, tanimura2020NumericallyExactApproach}. Despite this progress, RDM-based methods still face a severe curse of dimensionality for large systems or complex baths.

The second class is stochastic unravelling \cite{Breuer2007TheTheoryOfOQS}, which represents the reduced state by an ensemble of stochastic trajectories. In the Markovian regime, the stochastic Schrödinger equation \cite{Dalibard1992WavefunctionApproach, carmichael1993OpenSystemsApproach} derived from the Lindblad master equation \cite{lindblad1976GeneratorsQuantumDynamical} can be solved by standard SDE methods \cite{Brehier2023AnalysisOfASplittingScheme, huang2026ConformalStructurepreservingMethods}. For general non-Markovian dynamics, however, the NMSSE \cite{Diosi1997NMSSE, NMQSD} is difficult to solve directly. Existing strategies include time-local formulations based on the $O$-operator \cite{NMQSD, Yu1999NonMarkovianQSDPerturbativeApproach} and the hierarchy of pure states (HOPS) \cite{suess2014HierarchyStochasticPure}, but these methods are still constrained by the accuracy of the $O$-operator approximation or by the exponential decomposition of the BCF. A rigorous convergence analysis for hierarchical stochastic solvers also remains limited.

Motivated by these limitations, we propose a low-rank hierarchical method for the NMSSE. The method uses a low-rank approximation of the BCF to compress environmental memory, leading to a new infinite-dimensional hierarchy that generalizes HOPS. We then introduce a finite-order truncation and prove its convergence under the assumption that the NMSSE solution is unique. Numerical experiments on benchmark OQS models confirm the efficiency and accuracy of the proposed algorithm.

The remainder of this paper is organized as follows. Section \ref{sec:NMSSE} introduces the NMSSE. Section \ref{sec:method} develops the low-rank approximation induced hierarchical framework and its finite truncation. Section \ref{sec:convergence} establishes the convergence of the hierarchical framework, including the low-rank approximation and the finite truncation. Section \ref{sec:numerical_scheme} presents numerical schemes for the finite hierarchical system. Section \ref{sec:numerical_experiments} reports numerical experiments, and Section \ref{sec:conclusion} concludes the paper.%

\section{Preliminaries}
\label{sec:NMSSE}

Let $\mathcal{H}$ be a separable Hilbert space with inner product $\langle\cdot,\cdot\rangle_{\mathcal{H}}$ and induced norm $\|\cdot\|_{\mathcal{H}}$. Let $(\Omega,\mathcal{F},\mathbb{P})$ be a probability space.  We consider the following non-Markovian stochastic Schr\"{o}dinger equation \cite{Diosi1997NMSSE,NMQSD}:
\begin{equation}
\begin{aligned}
    \label{eq:NMSSE}
    \frac{\partial}{\partial t} \psi_t = \left(-iH+z_tL\right)\psi_t - L^{\dagger}\int_{0}^{t}\diff s\,\alpha(t,s)\frac{\delta}{\delta z_s}\psi_t, \quad t\in[0, T].
\end{aligned}
\end{equation}
The mathematical constituents are defined as follows:
\begin{itemize}
    \item $\alpha(\cdot,\cdot)\in \mathcal{C}([0,T]^2,\mathbb{C})$ is the two-time bath correlation function, which characterizes temporal correlations of the bath. It is Hermitian and positive semidefinite; i.e., for any $f\in L^2([0,T],\mathbb{C})$,
    \begin{equation*}
        \int_0^T\int_0^T f^*(u)\alpha(u,v)f(v)\diff u\diff v \geq 0.
    \end{equation*}
    \item $z\in L^2(\Omega; \mathcal{C}([0,T],\mathbb{C}))$ is a complex-valued centered Gaussian process over the probability space $(\Omega, \mathcal{F}, \mathbb{P})$ with covariance $\alpha$ and pseudo-covariance $0$, i.e.,
    \begin{equation}
        \mathbb{E}(z_t) = 0, \quad \mathbb{E}(z_tz_s) = 0,\quad \mathbb{E}(z_t^*z_s) = \alpha(t,s),\quad \forall t,s\in [0, T].
    \label{eq:noise_condition}
    \end{equation}
    Let $\mathcal{F}_t^z$ be the $\sigma$-algebra generated by $z$, i.e., $\mathcal{F}_t^z = \sigma(z_s: s\in[0,t])$.
    \item $H:D(H)\subset\mathcal{H}\to\mathcal{H}$, which represents the system Hamiltonian, is a self-adjoint linear operator such that $-iH$ serves as the infinitesimal generator of a unitary group $S(\cdot): S(t)=e^{-itH}$.
    \item $L\in\mathcal{L}(\mathcal{H})$ is the system-bath coupling operator, which is bounded with respect to $\|\cdot\|_{\mathcal{H}}$, i.e.,
    \begin{equation*}
        \|L\|_{\mathcal{H}} \coloneqq \sup_{\|\psi\|_{\mathcal{H}}= 1}\|L\psi\|_{\mathcal{H}} < \infty.
    \end{equation*}
    \item The initial value $\psi_0$ is an $\mathcal{H}$-valued random variable independent of $\mathcal{F}_T^z$. We consider the solution $\psi_t$ adapted to the filtration $\mathcal{F}_t^{z,\psi_0}$ generated by $\psi_0$ and $z_s, s\in[0,t]$.
    \item For the last term on the right-hand side of \eqref{eq:NMSSE}, we denote
    \begin{equation*}
        \mathcal{R}_t\psi_t \coloneq \int_{0}^{t} \diff s\, \alpha(t,s)\frac{\delta}{\delta z_s}\psi_t,
    \end{equation*}
    where $\mathcal{R}_t : D(\mathcal{R}_t)\subset\mathcal{Y}_t\to\mathcal{H}$ is a linear operator, and $\mathcal{Y}_t$ denotes the set of all $\mathcal{F}_t^{z,\psi_0}$-measurable $\mathcal{H}$-valued random variables over $(\Omega, \mathcal{F}, \mathbb{P})$.
    The term $\mathcal{R}_t\psi_t$ 
    is interpreted in the sense of the G\^{a}teaux derivative, i.e., for any $\omega\in\Omega$,
    \begin{equation*}
        \mathcal{R}_t\psi_t(\omega) = \left.\frac{\mathrm{d}}{\mathrm{d}\varepsilon}\mathscr{M}_t\left(\psi_0(\omega),z(\omega)|_{[0,t]} + \varepsilon\alpha(t,\cdot)|_{[0,t]}\right)\right|_{\varepsilon=0},
    \end{equation*}
    where 
    $\mathscr{M}_t: \mathcal{H}\times \mathcal{C}([0,t],\mathbb{C}) \to \mathcal{H}$ is a mapping such that 
    \begin{equation*}
        \psi_t(\omega) = \mathscr{M}_t\left(\psi_0(\omega),z(\omega)|_{[0,t]}\right)\qquad \forall\, \omega\in \Omega.
    \end{equation*}
    The existence of $\mathscr{M}_t$ is guaranteed by the Doob-Dynkin lemma \cite{Kallenberg2021FoundationsOf}.
\end{itemize}

An $\mathcal{F}_t^{z,\psi_0}$-adapted $\mathcal{H}$-valued process $\psi_t$, $t\in[0,T]$, is a \emph{mild solution} \cite{Prato2014StochasticEquations} to \eqref{eq:NMSSE} if $\psi_t$ takes values in $D(\mathcal{R}_t)$, a.s.,
\begin{subequations}
\begin{align}
    \label{eq:mild_condition_a}
    & \mathbb{P}\left(\int_{0}^{T}\left(1+|z_t|\|L\|_{\mathcal{H}}\right)\|\psi_t\|_{\mathcal{H}}\diff t < \infty\right) = 1, \\
    \label{eq:mild_condition_b}
    & \mathbb{P}\left(\int_{0}^{T}\left\|\mathcal{R}_t\psi_t\right\|_{\mathcal{H}}\diff t < \infty\right) = 1,
\end{align}
\end{subequations}
and for arbitrary $t\in[0,T]$, it holds a.s. that
\begin{equation}
    \label{eq:mild_formulation}
    \psi_t = S(t)\psi_0 + \int_0^t S(t-s) \left(z_sL - L^{\dagger}\mathcal{R}_s\right)\psi_s \diff s.
\end{equation}

\begin{remark}
\label{rmk:strong_mild}
    Let $\widetilde{\psi}_t = S(-t)\psi_t$ and $\widetilde{L}(s) = S(-s)L S(s)$ denote the interaction picture of the mild solution $\psi_t$ and the operator $L$, respectively. Then it holds that 
    \begin{equation}
    \label{eq:strong_solution}
        \widetilde{\psi}_t = \widetilde{\psi}_0 + \int_0^t \left(z_s\widetilde{L}(s) - \widetilde{L}^{\dagger}(s)\mathcal{R}_s\right)\widetilde{\psi}_s \diff s,
    \end{equation}
    which follows from the fact that $\mathcal{R}_s$ commutes with $S(\cdot)$. This identity indicates that the mild solution of \eqref{eq:NMSSE} is equivalent to the strong solution of the NMSSE in the interaction picture:
    \begin{equation}
    \label{eq:NMSSE_interaction}
        \frac{\partial}{\partial t}\widetilde{\psi}_t = \left(z_t\widetilde{L}(t) - \widetilde{L}^{\dagger}(t)\mathcal{R}_t\right)\widetilde{\psi}_t.
    \end{equation}
    In what follows, we shall exploit this equivalence to analyze the strong solution of the interaction picture equation whenever it is notationally or analytically advantageous.
\end{remark}


The main numerical challenge in the NMSSE is the functional-derivative integral together with colored noise. This feature prevents a direct application of standard solvers for ordinary differential equations and stochastic differential equations, and makes direct simulation of non-Markovian dynamics computationally difficult. To address this issue, existing approaches usually approximate or reformulate the functional derivative, thereby replacing the original temporally nonlocal problem with a tractable evolution system. In many cases, this corresponds to embedding the non-Markovian dynamics into an enlarged time-local, or effectively Markovian, system.


However, existing methods often impose restrictive assumptions on the bath correlation function or lack rigorous convergence guarantees. Motivated by these limitations, we develop a generalized hierarchical framework based on a low-rank approximation of the BCF, and we provide a rigorous convergence analysis to establish its theoretical validity.

\section{A low-rank hierarchical framework}
\label{sec:method}
\subsection{Infinite hierarchy}
We begin by approximating the bath correlation function through a low-rank approximation. Specifically,
\begin{equation}
    \alpha(t,s) \approx \sum_{j=1}^{r} \lambda_j V_j^*(t)V_j(s),\qquad \forall t,s\in[0,T],
\label{eq:low_rank_decomp_corr_func}
\end{equation}
where $\lambda_1\geq\cdots\geq\lambda_r\geq0$ and $V_j\in L^2([0,T],\mathbb{C})$ for $j=1,\cdots,r$. The existence of such an approximation follows from the positive semidefiniteness of the BCF and Mercer's theorem \cite{Mercer1909FunctionsOfPositive,Rasmussen2005GaussianProcess,Scholkopf2001LearningWithKernels,Riesz2012functional}. In practice, we compute this representation via singular value decomposition (SVD) or eigenvalue decomposition of the discretized bath correlation matrix $A = \{\alpha(t_i,t_j)\}$, and retain the leading $r$ components.

Substituting \eqref{eq:low_rank_decomp_corr_func} into \eqref{eq:NMSSE}, we obtain
\begin{equation}
    \frac{\partial}{\partial t} \psi_t = \left(-iH+z_tL\right)\psi_t - L^{\dagger}\sum_{j=1}^{r}\lambda_jV_j^*(t)\int_{0}^{t}\diff sV_j(s)\frac{\delta}{\delta z_s}\psi_t.
    \label{eq:low_rank_decomp_NMSSE}
\end{equation}
To derive a hierarchy following the HOPS formalism, we introduce the operators
\begin{equation}
    D_j = \int_{0}^{\infty} \diff sV_j(s)\frac{\delta}{\delta z_s}, \qquad j = 1,2,\cdots,r.
\label{eq:operator_Dj}
\end{equation}
Although the upper limit of integration in \eqref{eq:operator_Dj} is extended to $\infty$ in contrast to $t$ in \eqref{eq:low_rank_decomp_NMSSE}, the following identity holds:
\begin{equation*}
    D_j \psi_t = \int_{0}^{\infty} \diff sV_j(s)\frac{\delta}{\delta z_s}\psi_t = \int_{0}^{t} \diff sV_j(s)\frac{\delta}{\delta z_s}\psi_t.
\end{equation*}
This equivalence follows from the fact that $\psi_t$ is $\mathcal{F}_t^{z,\psi_0}$-adapted. Next, let $\boldsymbol{k}\in\mathbb{N}^r$ be a multi-index whose magnitude is defined by $|\boldsymbol{k}|\coloneqq\sum_{j=1}^rk_j$. We define the auxiliary wavefunctions recursively as
\begin{equation*}
    \psi_t^{(\boldsymbol{0})}\coloneqq \psi_t, \qquad \psi^{(\boldsymbol{e}_j)}_t\coloneqq D_j\psi_t^{(\boldsymbol{0})}, \qquad \psi_t^{(\boldsymbol{k}+\boldsymbol{e}_j)} \coloneqq D_j\psi_t^{(\boldsymbol{k})},
\end{equation*}
where $\boldsymbol{e}_j$ denotes the unit vector with a $1$ at the $j$th position and $0$ elsewhere. Consequently, the evolution equation for the primary wavefunction $\psi_t^{(\boldsymbol{0})}$ is given by
\begin{equation*}
    \frac{\partial}{\partial t}\psi_t^{(\boldsymbol{0})} = (-iH + z_tL)\psi_t^{(\boldsymbol{0})} - L^{\dagger}\sum_{j=1}^{r}\lambda_jV_j^*(t)\psi_t^{(\boldsymbol{e}_j)}.
\end{equation*}

We now derive the evolution equation for $\psi_t^{(\boldsymbol{e}_j)}$. Differentiating in time and applying the chain rule gives
\begin{subequations}
\begin{align}
    \label{eq:deriv_hier_1}
    \frac{\partial}{\partial t}\psi_t^{(\boldsymbol{e}_j)} &= \frac{\partial}{\partial t}\left(D_j\psi_t^{(\boldsymbol{0})}\right) \\
    &= \frac{\partial D_j}{\partial t}\psi_t^{(\boldsymbol{0})} + D_j\frac{\partial\psi_t^{(\boldsymbol{0})}}{\partial t} \\
    &= D_j\left[\left(-iH+z_tL\right)\psi_t^{(\boldsymbol{0})} - L^{\dagger}\sum_{k=1}^{r}\lambda_kV_k^*(t)\psi_t^{(\boldsymbol{e}_k)}\right] \\
    &= \left(-iH+z_tL\right)\psi_t^{(\boldsymbol{e}_j)} + V_j(t)L\psi_t^{(\boldsymbol{0})} - L^{\dagger}\sum_{k=1}^{r}\lambda_kV^*_k(t)\psi_t^{(\boldsymbol{e_j}+\boldsymbol{e}_k)},
    \label{eq:deriv_hier_4}
\end{align}
\end{subequations}
In \eqref{eq:deriv_hier_4}, we have utilized the fact that $D_j$ commutes with both $H$ and $L$, along with the identity $D_jz_t = V_j(t)$. This relation demonstrates that the evolution of $\psi_t^{(\boldsymbol{e}_j)}$ depends on the higher-order auxiliary wavefunctions $\psi_t^{(\boldsymbol{e}_j+\boldsymbol{e}_k)}$. By induction, we can present the coupled hierarchy for an arbitrary multi-index $\boldsymbol{k}\in\mathbb{N}^r$:
\begin{equation*}
    \frac{\partial}{\partial t} \psi_t^{(\boldsymbol{k})} = (-iH+z_tL)\psi_t^{(\boldsymbol{k})} + L\sum_{j=1}^{r}k_jV_j(t)\psi_t^{(\boldsymbol{k}-\boldsymbol{e}_j)} - L^{\dagger}\sum_{j=1}^{r}\lambda_jV_j^*(t)\psi_t^{(\boldsymbol{k}+\boldsymbol{e}_j)},
\end{equation*}
where $\psi_t^{(\boldsymbol{k})}\equiv 0$ for $\boldsymbol{k}\notin\mathbb{N}^r$.

\subsection{Finite truncation}
To obtain a computationally tractable model, we truncate the infinite hierarchy at a finite order. In contrast to the various truncation schemes developed for HOPS \cite{suess2014HierarchyStochasticPure}, we adopt a straightforward zero-truncation strategy; a detailed convergence analysis of this approach is provided in the subsequent section. Let $N$ denote the truncation order. The state of the system is then described by the set of auxiliary wavefunctions
\[
    \{\psi_t^{(\boldsymbol{k})}:\boldsymbol{k}\in\mathcal{K}^r_N\},
\]
where $\mathcal{K}^r_N\coloneqq\{\boldsymbol{k}\in\mathbb{N}^r:|\boldsymbol{k}|\leq N\}$. By imposing $\psi_t^{(\boldsymbol{k})}\equiv 0$ for $\boldsymbol{k}\notin\mathcal{K}^r_N$, we obtain the closed system
\begin{equation}
    \frac{\partial}{\partial t} \psi_t^{(\boldsymbol{k})} = (-iH+z_tL)\psi_t^{(\boldsymbol{k})} + L\sum_{\substack{j=1 \\ \boldsymbol{k}-\boldsymbol{e}_j\in\mathcal{K}^r_N}}^{r}k_jV_j(t)\psi_t^{(\boldsymbol{k}-\boldsymbol{e}_j)} - L^{\dagger}\sum_{\substack{j=1\\ \boldsymbol{k}+\boldsymbol{e}_j\in\mathcal{K}^r_N}}^{r}\lambda_jV_j^*(t)\psi_t^{(\boldsymbol{k}+\boldsymbol{e}_j)},
    \label{eq:finite_hierarchy}
\end{equation}
for $\boldsymbol{k}\in\mathcal{K}^r_N$, with the initial condition
\begin{equation*}
    \psi_0^{(\boldsymbol{0})} =\psi_0,\quad \text{and} \quad  \psi_0^{(\boldsymbol{k})}=0,\quad \forall\boldsymbol{k}\in\mathcal{K}^r_N, |\boldsymbol{k}|>0.
\end{equation*}


\section{Convergence of the hierarchical framework}
\label{sec:convergence}

In this section, we establish convergence of the proposed hierarchical framework via a two-stage argument. We first introduce the following uniqueness assumption:
\begin{enumerate}[label=(\Alph*), ref=(\Alph*), leftmargin=3em]
    \item \label{assum:exist_uniq}
    For every initial condition $\psi_0\in\mathcal{H}$, the NMSSE \eqref{eq:NMSSE} admits at most one mild solution on $[0,T]$. More precisely, if $\psi$ and $\phi$ are two $\mathcal{F}_t^{z,\psi_0}$-adapted $\mathcal{H}$-valued processes with continuous sample paths satisfying \eqref{eq:mild_condition_a}, \eqref{eq:mild_condition_b} and \eqref{eq:mild_formulation} with the same initial condition and driven by the same noise, then
    \begin{equation*}
        \mathbb{P}\left(\sup_{t\in[0,T]}\left\|\psi_t-\phi_t\right\|_{\mathcal{H}}=0\right) = 1.
    \end{equation*}
\end{enumerate}
Existence of the mild solution will be established via an infinite-series construction in the following analysis. The main result is stated in the next theorem:
\begin{theorem}\label{thm:main_convergence}
    Let $\sum_{j=1}^{r}\lambda_jV_j^*(t)V_j(s)$ be a rank-$r$ approximation of $\alpha(t,s)$. Under Assumption \ref{assum:exist_uniq}, there exist 
    $r_0>0$ such that for all $r>r_0$ and $N>0$,
    \begin{equation*}
       \mathbb{E}\left(\sup_{t\in[0,T]}\left\|\psi_t - \psi_t^{r,N}\right\|_{\mathcal{H}}\right) \leq C_1\varepsilon_r + C_2\sqrt{\frac{(\beta r)^N}{N!}},
    \end{equation*}
    where $C_1, C_2, \beta>0$ are generic constants which only depend on $\alpha$ and $T$, $\psi_t$ is the mild solution of \eqref{eq:NMSSE}, $\psi_t^{r,N}$ denotes the mild solution $\psi_t^{(\boldsymbol{0})}$ of \eqref{eq:finite_hierarchy} with rank $r$ and truncation order $N$, and 
    \begin{equation}
        \varepsilon_r \coloneq \sqrt{\sum_{j=r+1}^{\infty}\lambda_j^2} = o\left(r^{-\frac{1}{2}}\right).
    \label{eq:epsilon_r}
    \end{equation}
\end{theorem}

The proof of \cref{thm:main_convergence} is based on the decomposition
\begin{equation*}
    \mathbb{E}\left(\left\|\psi_t - \psi_t^{r,N}\right\|_{\mathcal{H}}\right) \leq \mathbb{E}\left(\left\|\psi_t - \psi_t^{r}\right\|_{\mathcal{H}}\right) + \mathbb{E}\left(\left\|\psi_t^{r} - \psi_t^{r,N}\right\|_{\mathcal{H}}\right),
\end{equation*}
where $\psi_t^{r}$ denotes the mild solution of \eqref{eq:low_rank_decomp_NMSSE}.
In the first part, we estimate the error between the exact NMSSE solution and the solution associated with the low-rank BCF approximation. Under Assumption \ref{assum:exist_uniq}, we show that this perturbation vanishes as the rank increases. In the second part, we analyze the truncation error induced by retaining only finitely many hierarchy levels, and prove convergence to the infinite-hierarchy solution as the truncation order tends to infinity. Combining the two estimates yields convergence of the proposed framework to the target dynamics. The specific details of these derivations are provided in the following two subsections.

\subsection{Convergence of the low-rank approximation}
To establish the convergence of the solution to the NMSSE under the low-rank BCF approximation, we first derive an infinite-series representation of the mild solution.

\begin{theorem}\label{thm:analytical_solution}
    The mild solution of \eqref{eq:NMSSE} is given by the infinite series
    \begin{equation}
    \begin{aligned}
        \psi_t =& S(t)\sum_{n=0}^{\infty}\sum_{m=0}^{\infty}\widetilde{\psi}_t^{(n,m)} \\
        \coloneq& S(t)\sum_{n=0}^{\infty}\sum_{m=0}^{\infty}\int_{0<\tau_1<\cdots<\tau_n<t}\diff^{n}\boldsymbol{\tau}z_{\tau_n}\cdots z_{\tau_1} \\
        &\qquad\quad \int_{0<s_1<\cdots<s_{2m}<t}\diff^{2m}\boldsymbol{s} (-1)^m\sum_{\boldsymbol{q}\in\mathscr{Q}(\boldsymbol{s})}\alpha(q_{2m},q_{2m-1})\cdots\alpha(q_2,q_1) \\
        &\qquad\quad \mathscr{T}\left\{\widetilde{L}(\tau_n)\cdots\widetilde{L}(\tau_1)\widetilde{L}^{\dagger}(q_{2m})\widetilde{L}(q_{2m-1})\cdots\widetilde{L}^{\dagger}(q_2)\widetilde{L}(q_1)\right\}\psi_0,
    \label{eq:infinite_series_solution}
    \end{aligned}
    \end{equation}
    where $\mathscr{T}$ is the time-ordering operator, $\widetilde{L}(\tau)\coloneqq S(-\tau)LS(\tau)$ is the interaction picture of $L$, and $\mathscr{Q}(\boldsymbol{s})=\mathscr{Q}(s_{1},\cdots,s_{2m})$ denotes the set of all possible ordered pairings of $\{s_1,\cdots,s_{2m}\}$. For example,
    \begin{equation*}
    \begin{aligned}
        \mathscr{Q}(s_1,s_2) =& \big\{\{(s_1,s_2)\}\big\}, \\
        \mathscr{Q}(s_1,s_2,s_3,s_4) =& \big\{\{(s_1,s_2),(s_3,s_4)\}, \{(s_1,s_3),(s_2,s_4)\}, \{(s_1,s_4),(s_2,s_3)\}\big\}.
    \end{aligned}
    \end{equation*}
    The first element of each pair is less than the second. If $\boldsymbol{q}$ is equal to some set of ordered pairings, it indicates that the elements of $\boldsymbol{q}$ are, in order, equal to all the elements in that set. For example, $\boldsymbol{q}=\{(s_1,s_4),(s_2,s_3)\}$ implies $q_1=s_1, q_2=s_4, q_3=s_2, q_4=s_3$.
    
    The infinite series \eqref{eq:infinite_series_solution} absolutely converges a.s. and has an upper bound
    \begin{equation}
    \begin{aligned}
        \left\|\psi_t\right\|_{\mathcal{H}}\leq& \mathcal{M}_t\|\psi_0\|_{\mathcal{H}}\\ 
        \coloneq& \exp\left\{\|L\|_{\mathcal{H}}\int_{0}^t|z_s|\diff s + \frac{\|L\|^2_{\mathcal{H}}}{2}\int_0^t\int_0^t|\alpha(u,v)|\diff u\diff v\right\}\|\psi_0\|_{\mathcal{H}}.
    \label{eq:upper_bound_inf_series}
    \end{aligned}
    \end{equation}
\end{theorem}
\begin{proof}
    The verification that \eqref{eq:infinite_series_solution} satisfies \eqref{eq:mild_condition_a}, \eqref{eq:mild_condition_b} and \eqref{eq:mild_formulation} is deferred to the supplementary materials. Here we prove absolute convergence of the series in \eqref{eq:infinite_series_solution}, i.e.,
    \begin{equation*}
        \sum_{n=0}^{\infty}\sum_{m=0}^{\infty}\left\|\widetilde{\psi}_t^{(n,m)}\right\|_{\mathcal{H}} < \infty
    \end{equation*}
    for $t\in[0,T]$ a.s.. Note that $\|\widetilde{L}(t)\|_{\mathcal{H}}=\|S(-t)LS(t)\|_{\mathcal{H}}=\|L\|_{\mathcal{H}}$. Hence,
    \begin{equation}
    \begin{aligned}
        \left\|\widetilde{\psi}_t^{(n,m)}\right\|_{\mathcal{H}} \leq& \int_{0<\tau_1<\cdots<\tau_n<t}\diff^n\boldsymbol{\tau}\left|z_{\tau_n}\cdots z_{\tau_1}\right|\int_{0<s_1<\cdots<s_{2m}<t}\diff^{2m}\boldsymbol{s} \\
        & \sum_{\boldsymbol{q}\in\mathscr{Q}(\boldsymbol{s})}\left|\alpha(q_{2m},q_{2m-1})\cdots\alpha(q_2,q_1)\right|\cdot \|L\|_{\mathcal{H}}^{n+2m} \cdot\|\psi_0\|_{\mathcal{H}}.
    \label{eq:absolute_converg_estimate}
    \end{aligned}
    \end{equation}
    We now rewrite the two integrals over the cubes $[0,t]^n$ and $[0,t]^{2m}$. By symmetry,
    \begin{equation}
        \int_{0<\tau_1<\cdots<\tau_n<t}\diff^n \boldsymbol{\tau}\left|z_{\tau_n}\cdots z_{\tau_1}\right|\cdot\|L\|_{\mathcal{H}}^n = \frac{1}{n!}\left(\|L\|_{\mathcal{H}}\int_0^t|z_s|\diff s\right)^n.
    \label{eq:4_8}
    \end{equation}
    Let $f(\boldsymbol{s})\coloneqq\sum_{\boldsymbol{q}\in\mathscr{Q}(\boldsymbol{s})}|\alpha(q_{2m},q_{2m-1})\cdots\alpha(q_2,q_1)|$. Since $\alpha(t,s)=\alpha^*(s,t)$, we have $|\alpha(t,s)|=|\alpha(s,t)|$, and thus $f(\boldsymbol{s})$ is symmetric under permutations of $s_1,\cdots,s_{2m}$. Because $\mathscr{Q}(s_1,\cdots,s_{2m})$ contains $(2m-1)!!$ elements, we obtain
    \begin{equation}
    \begin{aligned}
        \int_{0<s_1<\cdots<s_{2m}<t}\diff^{2m}\boldsymbol{s}f(\boldsymbol{s}) \|L\|_{\mathcal{H}}^{2m}
        =& \frac{1}{(2m)!}\int_{[0,t]^{2m}}\diff^{2m}\boldsymbol{s}f(\boldsymbol{s})\|L\|_{\mathcal{H}}^{2m} \\
        =& \frac{(2m-1)!!}{(2m)!}\left(\int_0^t\int_0^t\left|\alpha(u,v)\right|\diff u\diff v\right)^m\|L\|_{\mathcal{H}}^{2m} \\
        =& \frac{1}{m!}\left(\frac{\|L\|_{\mathcal{H}}^2}{2}\int_0^t\int_0^t|\alpha(u,v)|\diff u\diff v\right)^{m}.
    \label{eq:4_9}
    \end{aligned}
    \end{equation}
    Substituting \eqref{eq:4_8} and \eqref{eq:4_9} into \eqref{eq:absolute_converg_estimate} yields
    \begin{equation}
        \left\|\widetilde{\psi}_t^{(n,m)}\right\|_{\mathcal{H}} \leq \frac{1}{n!}\left(\|L\|_{\mathcal{H}}\int_0^t|z_s|\diff s\right)^n \frac{1}{m!}\left(\frac{\|L\|_{\mathcal{H}}^2}{2}\int_0^t\int_0^t|\alpha(u,v)|\diff u\diff v\right)^m  \|\psi_0\|_{\mathcal{H}}.
    \label{eq:4_10}
    \end{equation}
    Summing over $n,m\geq 0$ yields
    \begin{equation*}
    \begin{aligned}
        & \sum_{n=0}^{\infty}\sum_{m=0}^{\infty}\left\|\widetilde{\psi}_t^{(n,m)}\right\|_{\mathcal{H}} \\
        &\qquad\leq \exp\left\{\|L\|_{\mathcal{H}}\int_{0}^t|z_s|\diff s + \frac{\|L\|^2_{\mathcal{H}}}{2}\int_0^t\int_0^t|\alpha(u,v)|\diff u\diff v\right\}\|\psi_0\|_{\mathcal{H}} \\
        &\qquad< \infty,\quad \text{a.s.},
    \end{aligned}
    \end{equation*}
    which proves \eqref{eq:upper_bound_inf_series}.
\end{proof}

The following lemma provides an essential estimate for $\mathbb{E}(\exp\{\int_0^t|z_s|\diff s\})$, which will be utilized in the subsequent analysis.
\begin{lemma}\label{lemma:z_expect_estimate}
    Let $z$ be the Gaussian process satisfying \eqref{eq:noise_condition}. It holds that
    \begin{equation*}
        \mathbb{E}\left(\exp\left\{\int_0^t|z_s|\diff s\right\}\right) \leq \int_0^t \sqrt{\pi\alpha(s,s)}\exp\left\{\frac{\alpha(s,s)t^2}{4}\right\}\diff s + 1.
    \end{equation*}
\end{lemma}
\begin{proof}
    By the convexity of the exponential function, Jensen's inequality implies
    \begin{equation*}
        \exp\left\{\int_0^t|z_s|\diff s\right\} = \exp\left\{\frac{1}{t}\int_0^tt|z_s|\diff s\right\} \leq \frac{1}{t}\int_0^te^{t|z_s|}\diff s.
    \end{equation*}
    Taking expectation gives
    \begin{equation}
        \mathbb{E}\left(\exp\left\{\int_0^t|z_s|\diff s\right\}\right) \leq \frac{1}{t}\mathbb{E}\left(\int_0^te^{t|z_s|}\diff s\right) = \frac{1}{t}\int_0^t\mathbb{E}\left(e^{t|z_s|}\right)\diff s,
        \label{eq:estimate_414}
    \end{equation}
    where exchanging expectation and integration follows from Tonelli's theorem. Note that $|z_s|$ follows a Rayleigh distribution with density
    \begin{equation*}
        f(r) = \frac{2r}{\alpha(s,s)}e^{-r^2/\alpha(s,s)},\qquad r\geq 0.
    \end{equation*}
    A direct calculation gives
    \begin{equation}
    \begin{aligned}
    \label{eq:estimate_416}
        \mathbb{E}\left(e^{t|z_s|}\right) 
        =& 1+\sqrt{\pi\alpha(s,s)}te^{\alpha(s,s)t^2/4}\Phi\left(\sqrt{\frac{\alpha(s,s)}{2}}t\right) \\
        \leq& 1+\sqrt{\pi\alpha(s,s)}te^{\alpha(s,s)t^2/4},
    \end{aligned}
    \end{equation}
    where $\Phi$ is the cumulative distribution function of the standard normal distribution. Substituting \eqref{eq:estimate_416} into \eqref{eq:estimate_414} completes the proof.
\end{proof}

Let $\widehat{\alpha}_r(t,s)=\sum_{j=1}^r\lambda_jV_j^*(t)V_j(s)$ be the low-rank approximation of $\alpha(t,s)$ of rank $r$. The uniform convergence is given by the following lemma (see \cite{Riesz2012functional,Rasmussen2005GaussianProcess,Scholkopf2001LearningWithKernels}).
\begin{lemma}[Mercer's theorem]\label{lemma:Mercer}
    Let $\alpha(\cdot,\cdot)\in \mathcal{C}([0,T]^2,\mathbb{C})$ be positive semi-definite. Define the integral operator $\mathcal{I}:L^2([0,T],\mathbb{C})\to L^2([0,T],\mathbb{C})$ as 
    \begin{equation*}
        (\mathcal{I}f)(t) \coloneqq \int_0^T \alpha^*(t,s)f(s)\diff s.
    \end{equation*}
Let $V_j\in L^2([0,T],\mathbb{C})$ be the orthonormal eigenfunctions of $\mathcal{I}$ associated with the eigenvalues $\lambda_j\geq 0$, sorted in non-increasing order. Then
\begin{enumerate}
    \item the eigenvalues $\{\lambda_j\}_{j=1}^{\infty}$ are absolutely summable;
    \item $\alpha(t,s)=\sum_{j=1}^{\infty} \lambda_jV_j^*(t)V_j(s)$ holds for almost every $(t,s)\in[0,T]^2$, where the series converges absolutely and uniformly almost everywhere. And
    \begin{equation*}
        \int_0^T\int_0^T\Big|\alpha(t,s) - \sum_{j=1}^{r}\lambda_jV_j^*(t)V_j(s)\Big|^2\diff s\diff t = \sum_{j=r+1}^{\infty}\lambda_j^2.
    \end{equation*}
\end{enumerate}
\end{lemma}

Let $\psi_t$ and $\psi_t^r$ be the mild solutions of \eqref{eq:NMSSE} and \eqref{eq:low_rank_decomp_NMSSE}, respectively, with the same initial condition. In both equations, the noise $z$ is generated by $\alpha(t,s)$. The next theorem states convergence of the low-rank approximation.

\begin{theorem}\label{thm:low_rank_convergence}
    Under Assumption \ref{assum:exist_uniq}, there exists a constant $r_0>0$ such that for all $r>r_0$, we have
    \begin{equation}
        \mathbb{E}\left(\sup_{t\in[0,T]}\left\|\psi_t-\psi_t^{r}\right\|_{\mathcal{H}}\right) \leq C\varepsilon_r,
    \label{eq:converg_r}
    \end{equation}
    where $C$ is independent of $r$, and $\varepsilon_r$ is defined as \eqref{eq:epsilon_r}.
\end{theorem}
\begin{proof}
    It follows from \cref{lemma:Mercer} that
    \begin{equation*}
        \int_0^T\int_0^T\left|\alpha(t,s)-\widehat{\alpha}_r(t,s)\right|^2\diff t\diff s = \varepsilon_r^2.
    \end{equation*}
    Applying the Cauchy-Schwarz inequality, we obtain
    \begin{equation*}
        \int_0^T\int_0^T\left|\alpha(t,s)-\widehat{\alpha}_r(t,s)\right|\diff t\diff s \leq T\varepsilon_r.
    \end{equation*}
    According to \cref{thm:analytical_solution}, $\psi_t$ and $\psi_t^{r}$ admit the following series representations:
    \begin{equation*}
        \psi_t = S(t)\sum_{n=0}^{\infty}\sum_{m=0}^{\infty}\widetilde{\psi}_t^{(n,m)},
        \qquad
        \psi_t^{r} = S(t)\sum_{n=0}^{\infty}\sum_{m=0}^{\infty}\widetilde{\psi}_t^{r(n,m)}.
    \end{equation*}
    By exploiting the symmetry of the integrand, we have
    \begin{equation}
    \begin{aligned}
        g(t,\alpha,\widehat{\alpha}_r) &\coloneq \int_{0<s_1<\cdots<s_{2m}<t}\diff^{2m}\boldsymbol{s}\sum_{\boldsymbol{q}\in\mathscr{Q}(\boldsymbol{s})}\left|\prod_{j=1}^{m}\alpha(q_{2j},q_{2j-1})-\prod_{j=1}^{m}\widehat{\alpha}_{r}(q_{2j},q_{2j-1})\right| \\
        &= \frac{(2m-1)!!}{(2m)!} \int_{[0,t]^{2m}} \diff^{2m}\boldsymbol{s} \left|\prod_{j=1}^{m}\alpha(s_{2j},s_{2j-1})-\prod_{j=1}^{m}\widehat{\alpha}_{r}(s_{2j},s_{2j-1})\right|.
    \label{eq:4_22}
    \end{aligned}
    \end{equation}
    To estimate the integrand in \eqref{eq:4_22}, note that
    \begin{equation}
    \begin{aligned}
        & \left|\prod_{j=1}^{m}\alpha(s_{2j},s_{2j-1})-\prod_{j=1}^{m}\widehat{\alpha}_{r}(s_{2j},s_{2j-1})\right| \\ 
        \leq& \sum_{j=1}^{m} \prod_{k=j+1}^{m} |\widehat{\alpha}_r(s_{2k},s_{2k-1})| \cdot \left|\alpha(s_{2j},s_{2j-1})-\widehat{\alpha}_r(s_{2j},s_{2j-1})\right| \cdot \prod_{k=1}^{j-1} |\alpha(s_{2k},s_{2k-1})|.
    \label{eq:4_23}
    \end{aligned}
    \end{equation}
    Let $A\coloneq \int_0^T\int_0^T|\alpha(t,s)|\diff t\diff s$. Substituting \eqref{eq:4_23} into \eqref{eq:4_22} gives
    \begin{equation*}
    \begin{aligned}
        g(t,\alpha,\widehat{\alpha}_r) \leq g(T, \alpha, \widehat{\alpha}_r) 
        \leq \frac{1}{m!2^m}\left[\left(A+T\varepsilon_r\right)^m - A^m\right].
    \end{aligned}
    \end{equation*}
    Consequently, we have
    \begin{equation*}
    \begin{aligned}
        & \left\|\widetilde{\psi}_t^{(n,m)} - \widetilde{\psi}_t^{r(n,m)}\right\|_{\mathcal{H}} \\
        &\quad \leq \frac{\|L\|_{\mathcal{H}}^n}{n!}\left(\int_0^t|z_s|\diff s\right)^n \frac{\|L\|_{\mathcal{H}}^{2m}}{m!}\left[\left(\frac{A+T\varepsilon_r}{2}\right)^m - \left(\frac{A}{2}\right)^m\right] \|\psi_0\|_{\mathcal{H}}.
    \end{aligned}
    \end{equation*}
    Summing over $n$ and $m$, and invoking the absolute convergence of the series, we obtain
    \begin{equation*}
    \begin{aligned}
        \sup_{t\in[0,T]}\left\|\psi_t-\psi_t^{r}\right\|_{\mathcal{H}} 
        \leq& \sup_{t\in[0,T]}\sum_{n=0}^{\infty}\sum_{m=0}^{\infty}\left\|\widetilde{\psi}_t^{(n,m)} - \widetilde{\psi}_t^{r(n,m)}\right\|_{\mathcal{H}} \\
        \leq& e^{\|L\|_{\mathcal{H}}\int_{0}^T|z_s|\diff s}\left(e^{\frac{\|L\|_{\mathcal{H}}^2}{2}(A+T\varepsilon_r)} - e^{\frac{\|L\|_{\mathcal{H}}^2}{2}A}\right) \|\psi_0\|_{\mathcal{H}}.
    \end{aligned}
    \end{equation*}
    Finally, taking the expectation and applying the estimate from \cref{lemma:z_expect_estimate}, we establish the desired result \eqref{eq:converg_r}.
\end{proof}

\subsection{Convergence of the finite truncation}
\label{subsec:converg_trunc}
In this subsection, we establish convergence of the finitely truncated hierarchical system \eqref{eq:finite_hierarchy}. We equip $\mathbb{N}^r$ with the graded lexicographic order: for $\boldsymbol{j},\boldsymbol{k}\in\mathbb{N}^r$, we write $\boldsymbol{j}\prec\boldsymbol{k}$ if either $|\boldsymbol{j}|<|\boldsymbol{k}|$, or $|\boldsymbol{j}|=|\boldsymbol{k}|$ and $\boldsymbol{j}$ precedes $\boldsymbol{k}$ in the usual lexicographic order. Utilizing the operator $D_j$ defined in \eqref{eq:operator_Dj}, let $\Psi_t^r$ denote the infinite dimensional vector:
\begin{equation*}
    \Psi_t^r\coloneqq\left(\prod_{j=1}^r D_j^{k_j}\psi^r_t\right)_{\boldsymbol{k}\in\mathbb{N}^r},
\end{equation*}
where $\psi_t^r$ is the mild solution of \eqref{eq:low_rank_decomp_NMSSE}. To characterize the solution of the hierarchical system, we introduce the Banach space
\begin{equation*}
    \mathcal{H}_w \coloneqq \left\{\Psi=\left(\psi^{(\boldsymbol{k})}\right): \boldsymbol{k}\in\mathbb{N}^r,\ \ \psi^{(\boldsymbol{k})}\in\mathcal{H},\ \  \|\Psi\|_{w}<\infty\ \ \mathrm{a.s.}\right\},
\end{equation*}
with the norm $\|\cdot\|_w$ defined as
\begin{equation*}
    \left\|\Psi\right\|_w^2 \coloneqq \sum_{\boldsymbol{k}\in\mathbb{N}^r} \prod_{j=1}^{r}\frac{\lambda_{j}^{k_j}}{k_j!}\left\|\psi^{(\boldsymbol{k})}\right\|_\mathcal{H}^2,
\end{equation*}
where $\lambda_j$ correspond to the low-rank BCF approximation in \eqref{eq:low_rank_decomp_corr_func}. The following lemma shows that $\Psi_t^r\in\mathcal{H}_w$ for $t\in[0,T]$.
\begin{lemma}\label{lemma:estimate_D_j}
    The norm of each auxiliary component $\prod_{j=1}^rD_j^{k_j}\psi^r_t$ satisfies the following upper bound:
    \begin{equation}
        \left\|\prod_{j=1}^r D_j^{k_j}\psi^r_t\right\|_{\mathcal{H}} \leq \|L\|_{\mathcal{H}}^{|\boldsymbol{k}|}\prod_{j=1}^r\left(\int_0^t|V_j(s)|\diff s\right)^{k_j} \mathcal{M}_t^r \|\psi_0\|_{\mathcal{H}},
    \label{eq:estimate_D_j}
    \end{equation}
    where
    \begin{equation}
    \label{eq:M_t_r}
        \mathcal{M}_t^r \coloneq \exp\left\{\|L\|_{\mathcal{H}}\int_{0}^t|z_s|\diff s + \frac{\|L\|^2_{\mathcal{H}}}{2}\int_0^t\int_0^t|\widehat{\alpha}_r(u,v)|\diff u\diff v\right\}.
    \end{equation}
    Here, $\widehat{\alpha}_r$ is the rank-$r$ approximation of $\alpha$. It follows that $\Psi^r_t\in\mathcal{H}_w$, i.e., 
    \begin{equation*}
        \mathbb{P}\left(\left\|\Psi^r_t\right\|_{w}<\infty\right) = 1.
    \end{equation*}
\end{lemma}
\begin{proof}
    We consider the infinite series representation of $\psi_t^r$ in \cref{thm:analytical_solution}. We have an estimate of $D_j\widetilde{\psi}_t^{r(n,m)}$ following the procedure of \eqref{eq:absolute_converg_estimate}--\eqref{eq:4_10}:
    \begin{equation*}
    \begin{aligned}
        \left\|D_j\widetilde{\psi}_t^{r(n,m)}\right\|_{\mathcal{H}} \leq& \left(\|L\|_{\mathcal{H}}\int_0^t|V_j(s)|\diff s\right) \frac{1}{(n-1)!}\left(\|L\|_{\mathcal{H}}\int_0^t|z_s|\diff s\right)^{n-1} \\
        &\frac{1}{m!}\left(\frac{\|L\|_{\mathcal{H}}^2}{2}\int_0^t\int_0^t|\widehat{\alpha}_r(u,v)|\diff u\diff v\right)^m  \|\psi_0\|_{\mathcal{H}}.
    \end{aligned}
    \end{equation*}
    Summing them up yields
    \begin{equation*}
        \left\|D_j\psi_t^r\right\|_{\mathcal{H}} \leq \left(\|L\|_{\mathcal{H}}\int_0^t|V_j(s)|\diff s\right) \mathcal{M}_t^r \|\psi_0\|_{\mathcal{H}}.
    \end{equation*}
    By recursively applying the operator $D_j$ and iterating the aforementioned procedure, we obtain the bound \eqref{eq:estimate_D_j}. Substituting the estimate into the norm $\|\cdot\|_w$ implies
    \begin{equation}
        \|\Psi_t^r\|_w^2 \leq \sum_{\boldsymbol{k}\in\mathbb{N}^r}\prod_{j=1}^{r}\frac{1}{k_j!}\left(\|L\|_{\mathcal{H}}\int_0^t\sqrt{\lambda_j}|V_j(s)|\diff s\right)^{2k_j} \left(\mathcal{M}_t^r\right)^2\|\psi_0\|_{\mathcal{H}}^2.
    \label{eq:4_29}
    \end{equation}
    Furthermore, an application of the Cauchy-Schwarz inequality yields
    \begin{equation}
    \begin{aligned}
        \left(\int_0^t\sqrt{\lambda_j}|V_j(s)|\diff s\right)^2 
        \leq \int_0^t1^2\diff s\int_0^t\lambda_jV_j^*(s)V_j(s)\diff s 
        \leq t\int_0^t\alpha(s,s)\diff s.
    \label{eq:4_30}
    \end{aligned}
    \end{equation}
    Combining \eqref{eq:4_29} and \eqref{eq:4_30}, we arrive at
    \begin{equation*}
        \|\Psi_t^r\|_w^2 \leq \sum_{\boldsymbol{k}\in\mathbb{N}^r}\prod_{j=1}^{r}\frac{1}{k_j!}\left(t\|L\|_{\mathcal{H}}^2\int_0^t\alpha(s,s)\diff s\right)^{k_j} \left(\mathcal{M}_t^r\right)^2\|\psi_0\|_{\mathcal{H}}^2 < \infty, \quad \text{a.s.},
    \end{equation*}
    which completes the proof.
\end{proof}

For the following analysis, we now turn to the interaction picture of the truncated hierarchical system \eqref{eq:finite_hierarchy}, which reads
\begin{equation}
\label{eq:finite_hierarchy_interaction}
    \frac{\partial}{\partial t} \widetilde{\psi}_t^{(\boldsymbol{k})} = z_t\widetilde{L}(t)\widetilde{\psi}_t^{(\boldsymbol{k})} + \widetilde{L}(t)\sum_{\substack{j=1 \\ \boldsymbol{k}-\boldsymbol{e}_j\in\mathcal{K}^r_N}}^{r}k_jV_j(t)\widetilde{\psi}_t^{(\boldsymbol{k}-\boldsymbol{e}_j)} - \widetilde{L}^{\dagger}(t)\sum_{\substack{j=1\\ \boldsymbol{k}+\boldsymbol{e}_j\in\mathcal{K}^r_N}}^{r}\lambda_jV_j^*(t)\widetilde{\psi}_t^{(\boldsymbol{k}+\boldsymbol{e}_j)}.
\end{equation}
Let $\widetilde{\Psi}_t^{r,N}$ denote the strong solution to \eqref{eq:finite_hierarchy_interaction} subject to the initial conditions
\begin{equation*}
    \widetilde{\psi}_0^{(\boldsymbol{0})} = \widetilde{\psi}_0 = \psi_0, \quad \text{and}\quad \widetilde{\psi}_0^{(\boldsymbol{k})} = 0,\quad \forall \boldsymbol{k}\in\mathcal{K}^r_N, |\boldsymbol{k}|>0.
\end{equation*}
We embed $\widetilde{\Psi}_t^{r,N}$ into $\mathcal{H}_w$ by defining its components as $\widetilde{\psi}_t^{(\boldsymbol{k})}$ for $\boldsymbol{k} \in \mathcal{K}_N^r$ and setting $\widetilde{\psi}_t^{(\boldsymbol{k})} = 0$ for all $\boldsymbol{k} \notin \mathcal{K}_N^r$. Consequently, to establish the convergence of $\psi_t^{r,N}$ to $\psi_t^r$, it suffices to show that $\widetilde{\Psi}_t^{r,N}$ converges to $\widetilde{\Psi}_t^r$ in $\mathcal{H}_w$ as $N\to\infty$.

The interaction picture of the infinite hierarchical system can be reformulated as
\begin{equation*}
    \frac{\partial}{\partial t}\widetilde{\Psi}_t^{r} = \mathcal{A}(t)\widetilde{\Psi}_t^r,
\end{equation*}
where $\mathcal{A}(t)$ serves as the infinitesimal generator of the underlying propagator. We further define the projection operator $P_N \in \mathcal{L}(\mathcal{H}_w)$, which maps the infinite hierarchy onto its finite-dimensional subspace by retaining components indexed by $\boldsymbol{k} \in \mathcal{K}_N^r$ and mapping all higher-order components to zero; specifically, 
\begin{equation*}
    (P_N\Psi)_{\boldsymbol{k}} = \begin{cases}
        \psi^{(\boldsymbol{k})}, & \boldsymbol{k}\in\mathcal{K}^r_N, \\
        0, & \boldsymbol{k}\notin\mathcal{K}^r_N.
    \end{cases}
\end{equation*}
The finite system \eqref{eq:finite_hierarchy_interaction} is reformulated as
\begin{equation}
    \frac{\partial}{\partial t} \widetilde{\Psi}_t^{r,N} = \mathcal{A}_N(t)\widetilde{\Psi}_t^{r,N},
\label{eq:4_34}
\end{equation}
where $\mathcal{A}_N(t) \coloneqq P_N\mathcal{A}(t)P_N$. Consider the error function $E_N(t)\coloneqq \widetilde{\Psi}_t^{r,N}-P_N\widetilde{\Psi}_t^r$ subject to $E_N(0)=0$. By construction, $E_N(t)$ satisfies
\begin{equation*}
\begin{aligned}
    \frac{\partial}{\partial t} E_N(t) =& \frac{\partial}{\partial t}\widetilde{\Psi}_t^{r,N} - \frac{\partial}{\partial t}\left(P_N\widetilde{\Psi}_t^{r}\right) \\
    =& \mathcal{A}_N(t)\widetilde{\Psi}_t^{r,N} - P_N\mathcal{A}(t)\widetilde{\Psi}_t^r \\
    =& \mathcal{A}_N(t)\left(\widetilde{\Psi}_t^{r,N}-P_N\widetilde{\Psi}_t^r\right) + \left(\mathcal{A}_N(t)P_N - P_N\mathcal{A}(t)\right)\widetilde{\Psi}_t^r \\
    =& \mathcal{A}_N(t)E_N(t) + P_N\mathcal{A}(t)(P_N-I)\widetilde{\Psi}_t^r,
\end{aligned}
\end{equation*}
where $I$ denotes the identity operator in $\mathcal{L}(\mathcal{H}_w)$ and the last step follows from $P_N^2=P_N$. To derive a bound for the error $E_N(t)$, we first establish a stability estimate for \eqref{eq:4_34} using an energy-based approach.

\begin{lemma}\label{lemma:energy_estimate}
    The strong solution of \eqref{eq:4_34} has an estimate
    \begin{equation}
        \left\|\widetilde{\Psi}_t^{r,N}\right\|_w \leq \left\|\widetilde{\Psi}_0^{r,N}\right\|_w \exp\left\{\|L\|_{\mathcal{H}}\int_0^t|z_s|\diff s\right\}.
    \label{eq:4_36}
    \end{equation}
\end{lemma}
\begin{proof}
    Taking the time derivative of $\frac{1}{2}\|\widetilde{\Psi}_t^{r,N}\|_w^2$ and substituting \eqref{eq:finite_hierarchy_interaction} yields
    \begin{equation}
    \begin{aligned}
        \frac{1}{2}\frac{\partial}{\partial t}\left\|\widetilde{\Psi}_t^{r,N}\right\|_w^2 =\,& \mathrm{Re}\Bigg\{\sum_{\boldsymbol{k}\in\mathcal{K}_N^r}\prod_{j=1}^{r}\frac{\lambda_j^{k_j}}{k_j!}\left\langle\widetilde{\psi}_t^{(\boldsymbol{k})}, z_t\widetilde{L}(t)\widetilde{\psi}_t^{(\boldsymbol{k})}\right\rangle_{\mathcal{H}} \\
        &+ \sum_{\boldsymbol{k}\in\mathcal{K}_N^r}\sum_{\substack{\ell=1\\ \boldsymbol{k}-\boldsymbol{e}_{\ell}\in\mathcal{K}_N^r}}^{r}\prod_{j=1}^{r}\frac{\lambda_j^{k_j}}{k_j!}\left\langle\widetilde{\psi}_t^{(\boldsymbol{k})}, k_{\ell}V_{\ell}(t)\widetilde{L}(t)\widetilde{\psi}_t^{(\boldsymbol{k}-\boldsymbol{e}_{\ell})}\right\rangle_{\mathcal{H}} \\
        &- \sum_{\boldsymbol{k}\in\mathcal{K}_N^r}\sum_{\substack{\ell=1\\ \boldsymbol{k}+\boldsymbol{e}_{\ell}\in\mathcal{K}_N^r}}^{r}\prod_{j=1}^{r}\frac{\lambda_{j}^{k_j}}{k_j!}\left\langle\widetilde{\psi}_t^{(\boldsymbol{k})},\lambda_{\ell}V_{\ell}^*(t)\widetilde{L}^{\dagger}(t)\widetilde{\psi}_t^{(\boldsymbol{k}+\boldsymbol{e}_{\ell})}\right\rangle_{\mathcal{H}}\Bigg\}.
    \label{eq:energy_estimate}
    \end{aligned}
    \end{equation}
    Note that $\forall\phi,\psi\in\mathcal{H}$, the following identity holds:
    \begin{equation*}
        \left\langle\phi, V_{\ell}^*(t)\widetilde{L}^{\dagger}(t)\psi\right\rangle_{\mathcal{H}} = \left\langle V_{\ell}(t)\widetilde{L}(t)\phi, \psi \right\rangle_{\mathcal{H}} = \left\langle\psi,V_{\ell}(t)\widetilde{L}(t)\phi\right\rangle_{\mathcal{H}}^*.
    \end{equation*}
    By relabeling the indices, the last two terms on the right-hand side of \eqref{eq:energy_estimate} vanish. Consequently, we obtain the following inequality:
    \begin{equation*}
    \begin{aligned}
        \frac{1}{2}\frac{\partial}{\partial t}\left\|\widetilde{\Psi}_t^{r,N}\right\|_w^2 \leq |z_t| \|L\|_{\mathcal{H}}\left\|\widetilde{\Psi}_t^{r,N}\right\|_w^2.
    \end{aligned}
    \end{equation*}
    Finally, an application of Gronwall's inequality yields the desired estimate \eqref{eq:4_36}.
\end{proof}

With the stability estimate from \cref{lemma:energy_estimate}, we now proceed to establish the convergence by invoking Duhamel's principle.
\begin{theorem}\label{thm:truncation_convergence}
    Let $\psi_t^{r,N}$ and $\psi_t^{r}$ denote the $\boldsymbol{0}$-th solution of \eqref{eq:finite_hierarchy} and the mild solution of \eqref{eq:low_rank_decomp_NMSSE}, respectively. There exist constants $C, \beta > 0$ independent of $r$ and $N$ such that
    \begin{equation}
    \label{eq:truncation_converg_order}
        \mathbb{E}\left(\sup_{t\in[0,T]}\left\|\psi_t^{r} - \psi_t^{r,N}\right\|_{\mathcal{H}}\right) \leq C\sqrt{\frac{(\beta r)^N}{N!}}.
    \end{equation}
\end{theorem}
\begin{proof}
    Let $\boldsymbol{v}^{s}$ denote the solution of the homogeneous equation $\frac{\partial}{\partial t}\boldsymbol{v}^s = \mathcal{A}_N(t)\boldsymbol{v}^s$ with the initial condition $\boldsymbol{v}^s(s) = P_N\mathcal{A}(s)(P_N-I)\widetilde{\Psi}_s^r$. By Duhamel's principle, we have $E_N(t) = \int_0^t \boldsymbol{v}^s(t)\diff s$. By \cref{lemma:energy_estimate}, we have
    \begin{equation*}
    \begin{aligned}
        \left\|\boldsymbol{v}^s(t)\right\|_w \leq& \exp\left\{\|L\|_{\mathcal{H}}\int_s^t|z_u|\diff u\right\} \left\|P_N\mathcal{A}(s)(P_N-I)\widetilde{\Psi}_s^r\right\|_w.
    \end{aligned}
    \end{equation*}
    Consequently,
    \begin{equation}
    \begin{aligned}
        \left\|E_N(t)\right\|_w \leq& \exp\left\{\|L\|_{\mathcal{H}}\int_0^t|z_s|\diff s\right\} \int_0^t \left\|P_N\mathcal{A}(s)(P_N-I)\widetilde{\Psi}_s^r\right\|_w \diff s.
    \label{eq:4_42}
    \end{aligned}
    \end{equation}
    It remains to bound $\left\|P_N\mathcal{A}(s)(P_N-I)\widetilde{\Psi}_s^r\right\|_w$. We observe that $P_N \mathcal{A}(s)(P_N - I) \widetilde{\Psi}_s^r$ vanishes except for those components associated with multi-indices $\boldsymbol{k}$ on the truncation boundary, satisfying $|\boldsymbol{k}| = N$. For any such multi-index, the corresponding $\boldsymbol{k}$-th component is given by
    \begin{equation*}
        \sum_{j=1}^r \lambda_jV_j^*(s)\widetilde{L}^{\dagger}(s)\widetilde{\psi}_s^{(\boldsymbol{k}+\boldsymbol{e}_j)} \coloneq \sum_{j=1}^r \lambda_jV_j^*(s)\widetilde{L}^{\dagger}(s) D_j\prod_{m=1}^r D_m^{k_m} \widetilde{\psi}_s^r.
    \end{equation*}
    By definition,
    \begin{equation}
        \int_0^t \left\|P_N\mathcal{A}(s)(P_N-I)\widetilde{\Psi}_s^r\right\|_w \diff s = \int_0^t\sqrt{\sum_{|\boldsymbol{k}|=N}\prod_{j=1}^r\frac{\lambda_j^{k_j}}{k_j!}\left\|\sum_{j=1}^r\lambda_jV_j^*(s)\widetilde{L}^{\dagger}(s)\widetilde{\psi}_s^{(\boldsymbol{k}+\boldsymbol{e}_j)}\right\|_{\mathcal{H}}^2}\diff s.
    \label{eq:4_44}
    \end{equation}
    Applying the Cauchy-Schwarz inequality and \cref{lemma:estimate_D_j}, we obtain:
    \begin{equation}
    \begin{aligned}
        \left\|\sum_{j=1}^r\lambda_jV_j^*(s)\widetilde{\psi}_s^{(\boldsymbol{k}+\boldsymbol{e}_j)}\right\|_{\mathcal{H}}^2
        \leq& \Bigg(\sum_{j=1}^r \lambda_jV_j^*(s)V_j(s)\Bigg)\cdot 
        \Bigg(\|L\|_{\mathcal{H}}^{|\boldsymbol{k}|+1}\mathcal{M}_t^r\|\psi_0\|_{\mathcal{H}} \\
        & \prod_{\ell=1}^r\left(\int_0^t|V_{\ell}(u)|\diff u\right)^{k_{\ell}}\sum_{j=1}^r\int_0^t\sqrt{\lambda_j}|V_j(u)|\diff u \Bigg)^2,
    \label{eq:4_45}
    \end{aligned}
    \end{equation}
    where $\mathcal{M}_t^r$ is defined in \eqref{eq:M_t_r}. Here, we have utilized the fact that $s\leq t$ and substituted $t$ into the final term on the right-hand side. Applying Cauchy-Schwarz again, together with $\sum_{j=1}^r\lambda_jV_j^*(s)V_j(s) \leq \alpha(s,s)$, we obtain
    \begin{equation}
        \sum_{j=1}^r \left(\int_0^t \sqrt{\lambda_j}|V_j(u)|\diff u\right)^2 \leq t\int_0^t\alpha(s,s)\diff s.
    \label{eq:4_46}
    \end{equation}
    Combining \eqref{eq:4_30}, \eqref{eq:4_42}, \eqref{eq:4_44}, \eqref{eq:4_45}, and \eqref{eq:4_46} yields
    \begin{equation}
    \begin{aligned}
    \label{eq:4_54}
        \sup_{t\in[0,T]}\left\|\psi_t^{r} - \psi_t^{r,N}\right\|_{\mathcal{H}} =& \sup_{t\in[0,T]}\left\|\widetilde{\psi}_t^{r} - \widetilde{\psi}_t^{r,N}\right\|_{\mathcal{H}} 
        \leq \sup_{t\in[0,T]}\left\|E_N(t)\right\|_w \\
        \leq& \exp\left\{\|L\|_{\mathcal{H}}\int_0^T|z_s|\diff s\right\} \|L\|_{\mathcal{H}}^2\mathcal{M}_T^r\|\psi_0\|_{\mathcal{H}}\int_0^T\sqrt{\alpha(s,s)}\diff s \\
        &\ T\int_0^T\alpha(s,s)\diff s \sqrt{\sum_{|\boldsymbol{k}|=N}\prod_{j=1}^r \frac{1}{k_j!}\left(T\|L\|_{\mathcal{H}}^2\int_0^T\alpha(s,s)\diff s\right)^{k_j}}.
    \end{aligned}
    \end{equation}
    By Cauchy-Schwarz inequality,
    \begin{equation}
    \begin{aligned}
    \label{eq:4_55}
        \int_0^T\int_0^T\left|\widehat{\alpha}_r(t,s)\right|\diff t\diff s \leq& \,T\sqrt{\int_0^T\int_0^T\left|\widehat{\alpha}_r(t,s)\right|^2\diff t\diff s} \\
        \leq& \,T\sqrt{\int_0^T\int_0^T\left|\alpha(t,s)\right|^2\diff t\diff s}.
    \end{aligned}
    \end{equation}
    Note that for any $\beta > 0$, it holds that
    \begin{equation*}
        \sum_{|\boldsymbol{k}|=N}\prod_{j=1}^r\frac{\beta^{k_j}}{k_j!} = \frac{(\beta r)^N}{N!}.
    \end{equation*}
    Substituting \eqref{eq:4_55} into $\mathcal{M}_T^r$ in \eqref{eq:4_54}, taking expectation on both sides of \eqref{eq:4_54}, and applying \cref{lemma:z_expect_estimate}, we complete the proof.
\end{proof}


\subsection{Proof of Theorem \ref{thm:main_convergence}}
The proof of \cref{thm:main_convergence} follows directly by synthesizing the results of \cref{thm:low_rank_convergence} and \cref{thm:truncation_convergence}.

\section{Implementation}
\label{sec:numerical_scheme}
In this section, we describe the numerical generation of the colored noise paths and the low-rank approximation of the BCF. Furthermore, we present the numerical schemes employed to solve the finitely truncated hierarchical system \eqref{eq:finite_hierarchy}.

\subsection{Generation of the noise path and eigenfunctions}
The colored noise path $z_t$ is generated by Cholesky factorization. We discretize $[0,T]$ using $t_i=i\Delta t$, $\Delta t=T/M$, and set $\boldsymbol{z}=(z_{t_0},\ldots,z_{t_M})^{\top}$. Its covariance matrix is $S_{ij}=\alpha^*(t_i,t_j)$. Given a factorization $S=LL^{\dagger}$, one sample path is obtained from
\[
    \boldsymbol{z}=L\boldsymbol{\xi}, \qquad \boldsymbol{\xi}\sim\mathcal{CN}(0,I_{M+1}).
\]

The eigenpairs of $\alpha$ are computed from the grid matrix $A_{ij}=\alpha(t_i,t_j)$. If $A=U\Sigma V^{\dagger}$ is its singular value decomposition, then the discrete eigenvalues are approximated by the diagonal entries of $\Delta t\,\Sigma$, while the sampled eigenfunctions are given by the columns of $V/\sqrt{\Delta t}$.

\subsection{Numerical scheme for the finite hierarchical system}
We employ a Strang splitting scheme to solve the finite hierarchical system \eqref{eq:finite_hierarchy}. This approach is motivated by the non-commutativity between the Hamiltonian $H$ (representing the conservative dynamics) and the operator $L$ (associated with the dissipative terms). Let $\Delta t > 0$ denote the time step, and define the discrete time points $t_n = n\Delta t$ for $n=0, 1, \dots$. The Strang splitting update for the state $\Psi^{r,N}_t$ over one time step $[t_n, t_{n+1}]$ is formulated as
\begin{equation}
\label{eq:strang_splitting}
\begin{array}{lll}
    \displaystyle \frac{\partial}{\partial t}u_t^{(\boldsymbol{k})} = -iHu_t^{(\boldsymbol{k})}, &\quad u_{t_n}^{(\boldsymbol{k})} = \psi_{t_n}^{(\boldsymbol{k})}, &\quad t\in\left[t_n,t_{n+1/2}\right], \\[2ex]
    \displaystyle \frac{\partial}{\partial t}v_t^{(\boldsymbol{k})} = f^{(\boldsymbol{k})}(t,\boldsymbol{v}_t), &\quad v_{t_{n}}^{(\boldsymbol{k})}=u_{t_{n+1/2}}^{(\boldsymbol{k})}, &\quad t\in\left[t_n, t_{n+1}\right], \\[2ex]
    \displaystyle \frac{\partial}{\partial t}\phi_t^{(\boldsymbol{k})} = -iH\phi_t^{(\boldsymbol{k})}, &\quad \phi_{t_{n+1/2}}^{(\boldsymbol{k})} = v_{t_{n+1}}^{(\boldsymbol{k})}, &\quad t\in[t_{n+1/2}, t_{n+1}],
\end{array}
\end{equation}
for $\boldsymbol{k}\in\mathcal{K}^r_N$, where
\begin{equation*}
    f^{(\boldsymbol{k})}(t,\boldsymbol{v}_t) = z_tLv_t^{(\boldsymbol{k})} + L\sum_{\substack{j=1 \\ \boldsymbol{k}-\boldsymbol{e}_j\in\mathcal{K}^r_N}}^{r}k_jV_j(t)v_t^{(\boldsymbol{k}-\boldsymbol{e}_j)} - L^{\dagger}\sum_{\substack{j=1 \\ \boldsymbol{k}+\boldsymbol{e}_j\in\mathcal{K}^r_N}}^{r}\lambda_jV_j^*(t)v_t^{(\boldsymbol{k}+\boldsymbol{e}_j)}.
\end{equation*}
The numerical solution at the $(n+1)$-th time step is obtained as $\psi_{t_{n+1}}^{(\boldsymbol{k})} = \phi_{t_{n+1}}^{(\boldsymbol{k})}$. Regarding the splitting steps defined in \eqref{eq:strang_splitting}, the first and third subproblems correspond to the unitary evolution under the system Hamiltonian, which is computed via the propagator $S(\Delta t/2)$. The second subproblem, which involves the dissipative and stochastic terms, is integrated using a second-order Runge-Kutta method. In our implementation, we specifically employ Heun's method, which is formulated as follows:
\begin{equation*}
\begin{aligned}
    w_1^{(\boldsymbol{k})} =& f^{(\boldsymbol{k})}(t_n,\boldsymbol{v}_{t_n}), \\
    w_2^{(\boldsymbol{k})} =& f^{(\boldsymbol{k})}\left(t_{n+1},\boldsymbol{v}_{t_n} + \Delta t \boldsymbol{w}_1\right), \\
    v_{t_{n+1}}^{(\boldsymbol{k})} =& v_{t_n}^{(\boldsymbol{k})} + \frac{\Delta t}{2}\left(w_1^{(\boldsymbol{k})} + w_2^{(\boldsymbol{k})}\right).
\end{aligned}
\end{equation*}
A key computational advantage of this scheme in the present context is that it circumvents the requirement for evaluating the stochastic path $z_t$ and the eigenfunctions $V_j(t)$ at intermediate time steps.

\section{Numerical experiments}
\label{sec:numerical_experiments}

In this section, we present several numerical experiments to validate the proposed method for the NMSSE. We consider the two-level spin-boson model coupled to a bosonic bath, which is a benchmark model commonly used in the study of open quantum systems. Consequently, the system Hilbert space is identified as $\mathcal{H}=\mathbb{C}^2$. For any vectors $\phi, \psi \in \mathbb{C}^2$, the inner product and its associated norm are now defined as
\begin{equation*}
    \langle \phi, \psi \rangle = \phi^{\dagger}\psi, \quad\mathrm{and}\quad \|\psi\| = \sqrt{\langle \psi, \psi \rangle}.
\end{equation*}
The expectation value of an observable $O\in\mathcal{L}(\mathbb{C}^2)$ at time $t$ is given by
\begin{equation*}
    \langle O \rangle_t \coloneq \frac{\mathbb{E}\left(\left\langle\psi_t,O\psi_t\right\rangle\right)}{\mathbb{E}\left(\left\langle\psi_t,\psi_t\right\rangle\right)},
\end{equation*}
where $\psi_t$ is the solution to the NMSSE. To evaluate the accuracy of the proposed method, we benchmark our results against a reference solution obtained via the Time-Evolving Matrix Product Operators (TEMPO) method \cite{strathearn2018EfficientNonMarkovianQuantum}, which serves as a high-precision numerical standard for non-Markovian dynamics.

\begin{remark}
    The decay of the eigenvalues in the low-rank approximation is governed by the regularity of the BCF: if $\alpha\in \mathcal{C}^p$, then $\lambda_n=o(n^{-(p+1)})$ as $n\to\infty$ \cite{reade1984EigenvaluesPositiveDefinite}. Hence smoother BCFs require smaller numerical ranks. For a spectral density $J$, the BCF is generally given by
    \begin{equation*}
        \alpha(t,s) = \int_0^{\infty} J(\omega) \left[ \coth\left(\frac{\beta\omega}{2}\right) \cos(\omega(t-s)) - i \sin(\omega(t-s)) \right] \mathrm{d}\omega,
    \end{equation*}
    where $\beta$ is the inverse temperature and
    \begin{equation*}
        J(\omega) = \eta \omega^s f(\omega/\omega_c).
    \end{equation*}
    Here $\eta>0$, $s>0$, $\omega_c$ is the cutoff frequency, and $f$ specifies the high-frequency cutoff, e.g., $f(x)=e^{-x}$ or $f(x)=(x^2+1)^{-1}$. For the exponential cutoff with $s\geq 1$, $\alpha$ is smooth and the eigenvalues decay rapidly, so a small rank is sufficient. For less regular choices such as the Drude--Lorentz cutoff, one may instead use a Matsubara sum-of-exponentials representation \cite{tanimura2020NumericallyExactApproach}; in this setting our relaxed formulation recovers HOPS, as discussed in Section \ref{subsec:experiment_2}.
\end{remark}

\subsection{Exponential cutoff BCF with Ohmic spectral density}
We first evaluate the performance of our method using a spin-boson model with the following configuration:
\begin{itemize}
    \item The system Hamiltonian is given by $H=\varepsilon\sigma_z + \sigma_x$, where $\sigma_z$ and $\sigma_x$ are the Pauli matrices.
    \item The Lindblad operator is given by $L=\sigma_z$.
    \item The bath correlation function is given by \cite{cai2020InchwormMonteCarlo}
    \begin{equation}
        \alpha(t,s) = \sum_{j=1}^J\frac{c_{j}^2}{2\omega_{j}}\left[\coth\left(\frac{\beta\omega_j}{2}\right)\cos\left(\omega_j(t-s)\right) - i\sin\left(\omega_j(t-s)\right)\right],
    \label{eq:BCF_Ohmic}
    \end{equation}
    where 
    \begin{equation*}
    \begin{aligned}
        \omega_j &= -\omega_c\log\left(1-\frac{j}{J}\left[1-\exp(-\omega_{\mathrm{max}}/\omega_c)\right]\right), \\
        c_j &= \omega_j\sqrt{\frac{\xi\omega_c}{J}\left[1-\exp(-\omega_{\mathrm{max}}/\omega_c)\right]}.
    \end{aligned}
    \end{equation*}
\end{itemize}
The physical parameters are set as $\beta=5$, $\omega_c=2.5$, and $\omega_{\mathrm{max}}=4\omega_c$, with the number of bath modes $J=200$. We conduct experiments by varying the energy bias $\varepsilon$ and the coupling strength $\xi$. The system is initialized in the state $\psi_0 = (1,0)^{\top}$. Numerical simulations are performed over the time interval $[0,5]$ with a uniform time step $\Delta t=0.05$.

\begin{remark}
    This BCF \eqref{eq:BCF_Ohmic} can be written as a sum of exponentials:
    \begin{equation*}
         \alpha(t,s) = \sum_{j=1}^J \frac{c_j^2}{4\omega_j}\left(g_je^{i\omega_j(t-s)}+h_je^{-i\omega_j(t-s)}\right),
    \end{equation*}
    where $g_j=\coth\left(\beta\omega_j/2\right)-1$ and $h_j=\coth\left(\beta\omega_j/2\right)+1$. However, it is formidable to solve this problem with HOPS, since there are 400 exponentials in the BCF, which leads to a multi-index of length 400 to represent the auxiliary wave functions. The hierarchical system is thus too large to be solved in practice. In contrast, the proposed method can solve this problem efficiently by using a low-rank approximation of the BCF. \cref{fig:BCF_approx} shows that the BCF \eqref{eq:BCF_Ohmic} can be well approximated by a low-rank approximation with rank $r=10$.
\end{remark}

\begin{figure}[htbp]
    \centering
    \includegraphics[width=0.5\textwidth]{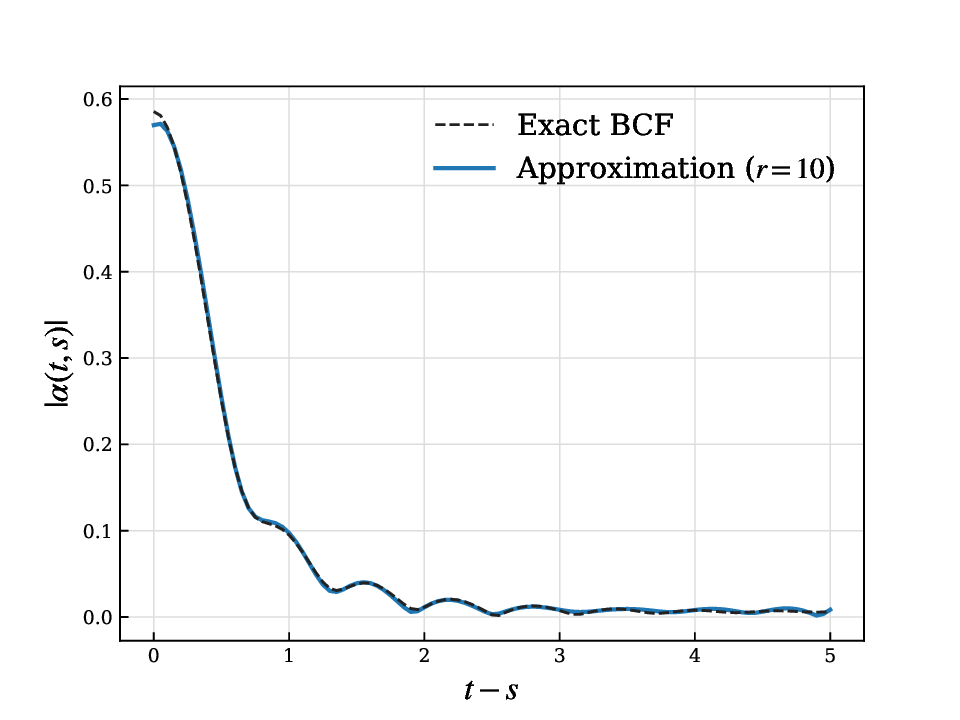}
    \caption{Amplitude of the BCF \eqref{eq:BCF_Ohmic} and its low-rank approximation with rank $r=10$.}
    \label{fig:BCF_approx}
\end{figure}

\Cref{fig:Ohmic_trajs_xi_multi} illustrates the evolution of $\langle\sigma_z\rangle_t$ for various values of $\varepsilon=0,1,2$, $\xi=0.2,0.4$, rank $r=10$, and truncation order $N=8$. Under this configuration, the total number of hierarchical equations is given by
\begin{equation*}
    \sum_{n=0}^N\binom{n+r-1}{r-1} = \binom{N+r}{r} = \binom{18}{10} = 43758.
\end{equation*}
This cardinality is significantly lower than that required by the standard HOPS method, which entails $\binom{408}{400}\approx 10^{15}$ equations. \Cref{fig:Ohmic_trajs_xi_multi} also compares the results obtained using different sample sizes for $z$. These numerical trajectories demonstrate that the proposed method converges to the reference solution as the sample size increases.

\begin{figure}[htbp]
    \centering
    \includegraphics[width=\textwidth]{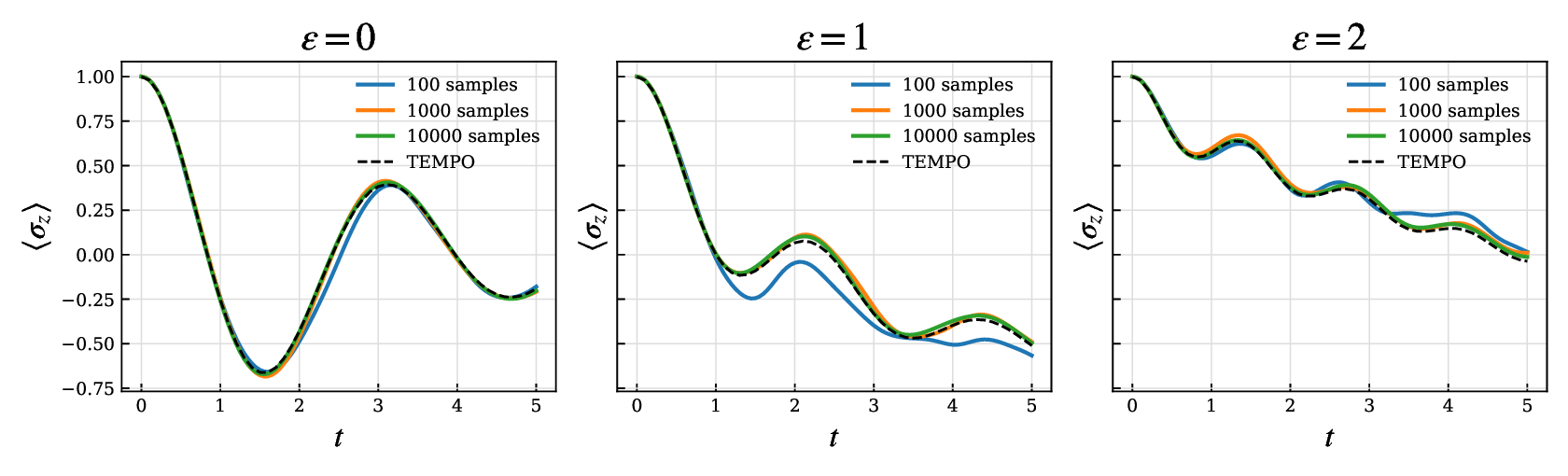}\\
    \includegraphics[width=\textwidth]{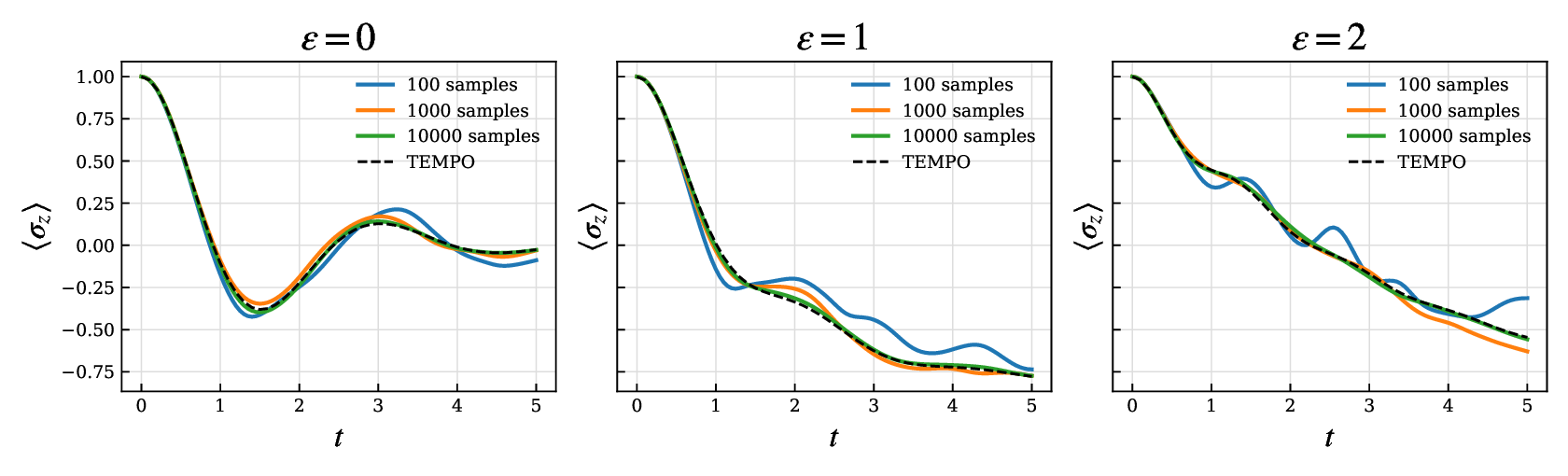}
    \caption{Evolution of $\langle\sigma_z\rangle_t$ for different parameter settings (top to bottom: $\xi=0.2,0.4$).}
    \label{fig:Ohmic_trajs_xi_multi}
\end{figure}




Since the size of the truncated hierarchy grows rapidly with the rank $r$, we restrict the convergence-order test to the truncation order $N$. For each fixed rank $r=2,4,8$, we vary $N$ and take the solution with $N=16$ as the reference solution, denoted by $\psi_t^{r,\mathrm{ref}}$. The error is measured by
\[
    e_N\coloneq \mathbb{E}\left(\sup_{t\in[0,T]}\left\|\psi_t^{r,\mathrm{ref}}-\psi_t^{r,N}\right\|\right).
\]
According to \eqref{eq:truncation_converg_order}, for fixed $r$ this error is expected to satisfy
\[
    e_N=\mathcal{O}\left(\sqrt{\frac{(\beta r)^N}{N!}}\right).
\]
To test this prediction, \Cref{fig:convergence_order} reports two normalized quantities. The first is $\log(e_N\sqrt{N!})$: after removing the factorial factor, the remaining dependence on $N$ should be at most exponential, so the curves in \Cref{fig:convergence_order_log} are expected to be approximately affine. The second diagnostic is the rescaled consecutive error ratio. From the same estimate, it should satisfy
\[
    \frac{e_{N+1}}{e_N}\sqrt{N+1}\approx \sqrt{\beta r},
\]
up to constants independent of $N$. Thus the bounded and stabilizing profiles in \Cref{fig:convergence_order_ratio} give a complementary check of the predicted order. In addition, the variation of the slopes in \Cref{fig:convergence_order_log} and of the limiting levels in \Cref{fig:convergence_order_ratio} with respect to $r$ is consistent with the $r$-dependent factor in \eqref{eq:truncation_converg_order}, further supporting the convergence estimate.

\begin{figure}[htbp]
    \centering
    \subfloat[$\log(e_N\sqrt{N!})$\label{fig:convergence_order_log}]{
        \includegraphics[width=0.4\textwidth]{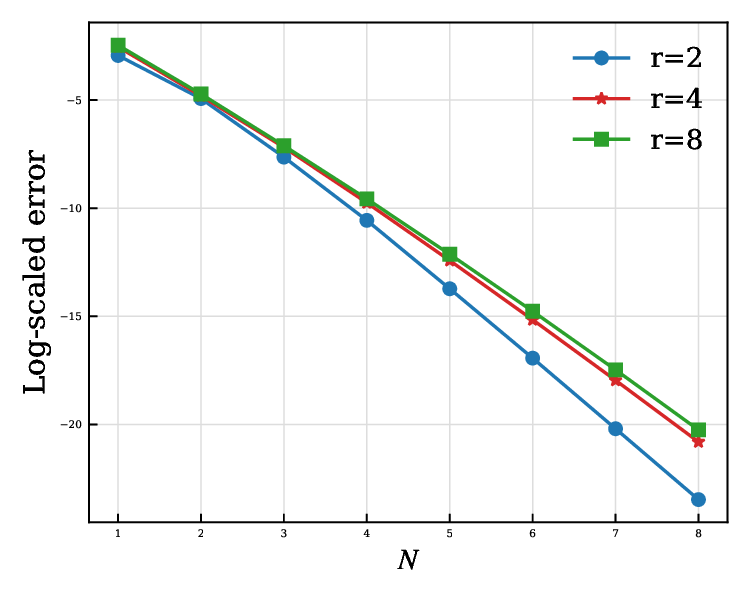}
    }
    \hspace{0.02\textwidth}
    \subfloat[$(e_{N+1}/e_N)\sqrt{N+1}$\label{fig:convergence_order_ratio}]{
        \includegraphics[width=0.4\textwidth]{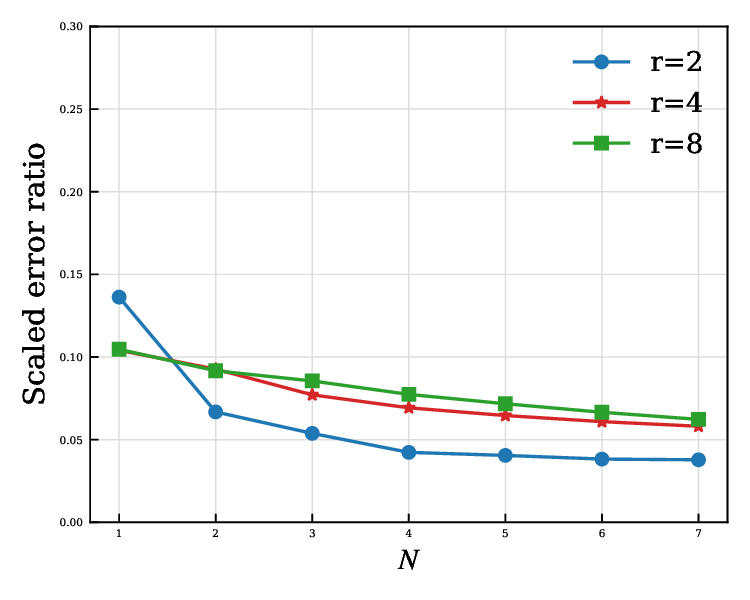}
    }
    \caption{Convergence-order tests with respect to the truncation order $N$ for fixed ranks $r=2,4,8$.}
    \label{fig:convergence_order}
\end{figure}

\subsection{Exponentially decaying BCF}
\label{subsec:experiment_2}
Next, we examine a spin-boson configuration characterized by an exponentially decaying BCF:
\begin{itemize}
    \item The system Hamiltonian is given by $H=\frac{1}{2}\sigma_z$.
    \item The Lindblad operator is given by $L=\sqrt{2}\sigma_z$.
    \item The bath correlation function is given by
    \begin{equation}
        \alpha(t,s) = \frac{\gamma}{2}e^{-\gamma|t-s|}.
    \label{eq:simple_bcf}
    \end{equation}
\end{itemize}
The system is initialized in the state $\psi_0 = (\sqrt{2}/2, \sqrt{2}/2)^{\top}$. Numerical simulations are performed over the time interval $[0,2]$ with a uniform time step $\Delta t=0.02$.

\begin{remark}
    Under this configuration, the problem has an exact solution. The reduced density matrix of the system at time $t$ is given by
    \begin{equation*}
        \rho(t) = \begin{pmatrix}
            \rho_{11}(0) & \rho_{12}(0)e^{-f(t)} \\
            \rho_{21}(0)e^{-f(t)^*} & \rho_{22}(0)
        \end{pmatrix},
    \end{equation*}
    where
    \begin{equation*}
        f(t) = \frac{4}{\gamma}(e^{-\gamma t}+\gamma t-1) + i t.
    \end{equation*}
    For any observable $O\in\mathcal{L}(\mathbb{C}^2)$, the expectation value can be computed as
    \begin{equation*}
        \langle O \rangle_t = \mathrm{Tr}\left(O\rho(t)\right).
    \end{equation*}
\end{remark}

\begin{remark}
    The BCF \eqref{eq:simple_bcf} is a single exponential corresponding to the simplest case for the HOPS method. Notably, this exponentially decaying BCF does not possess an intrinsic low-rank structure. \cref{fig:exp_BCF_approx} illustrates the amplitude of \eqref{eq:simple_bcf} with $\gamma=4$ and its low-rank approximation at ranks $r=10, 50$. Clearly, a high rank is required to achieve a sufficient approximation. Nevertheless, our proposed framework remains applicable by relaxing the approximation requirements. Specifically, we represent the BCF \eqref{eq:simple_bcf} using two functions $U,V\in L^2([0,T],\mathbb{C})$ such that
    \begin{equation}
        \alpha(t,s) = \lambda U(t)V(s), \qquad t\geq s,
    \label{eq:relax_approx}
    \end{equation}
    with $\lambda=\gamma/2$, $U(t) = e^{-\gamma t}$, and $V(t) = e^{\gamma t}$. Substituting \eqref{eq:relax_approx} into the NMSSE \eqref{eq:NMSSE} and following the procedure established in Section \ref{sec:method}, we obtain the hierarchical system
    \begin{equation}
        \frac{\partial}{\partial t}\psi_t^{(k)} = (-iH+z_tL)\psi_t^{(k)} + k\lambda V(t)L\psi_t^{(k-1)} - U(t)L^{\dagger}\psi_t^{(k+1)}
    \label{eq:relax_hierarchy}
    \end{equation}
    for $k=0,1,\cdots,N$ with $\psi_t^{(-1)} \equiv \psi_t^{(N+1)} \equiv 0$. This system is equivalent to the one derived via the standard HOPS method; see \autoref{sec:appendix} for a formal proof.
\end{remark}

\begin{figure}[htbp]
    \centering
    \includegraphics[width=0.5\textwidth]{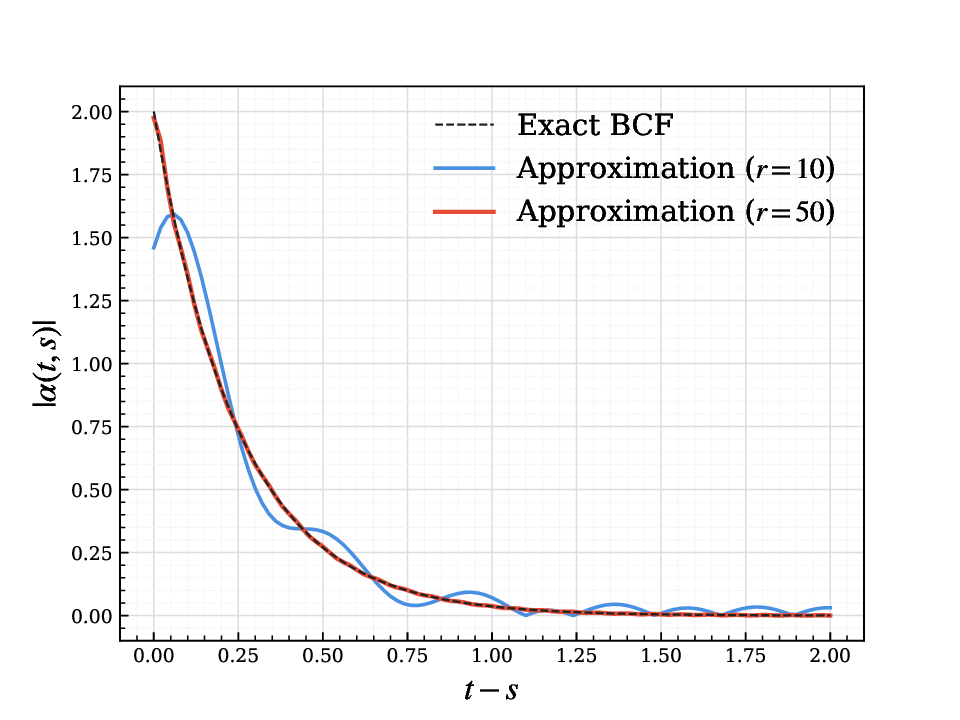}
    \caption{Amplitude of the BCF \eqref{eq:simple_bcf} with $\gamma=4$ and its low-rank approximation.}
    \label{fig:exp_BCF_approx}
\end{figure}

\begin{figure}[htbp]
    \centering
    \includegraphics[width=\textwidth]{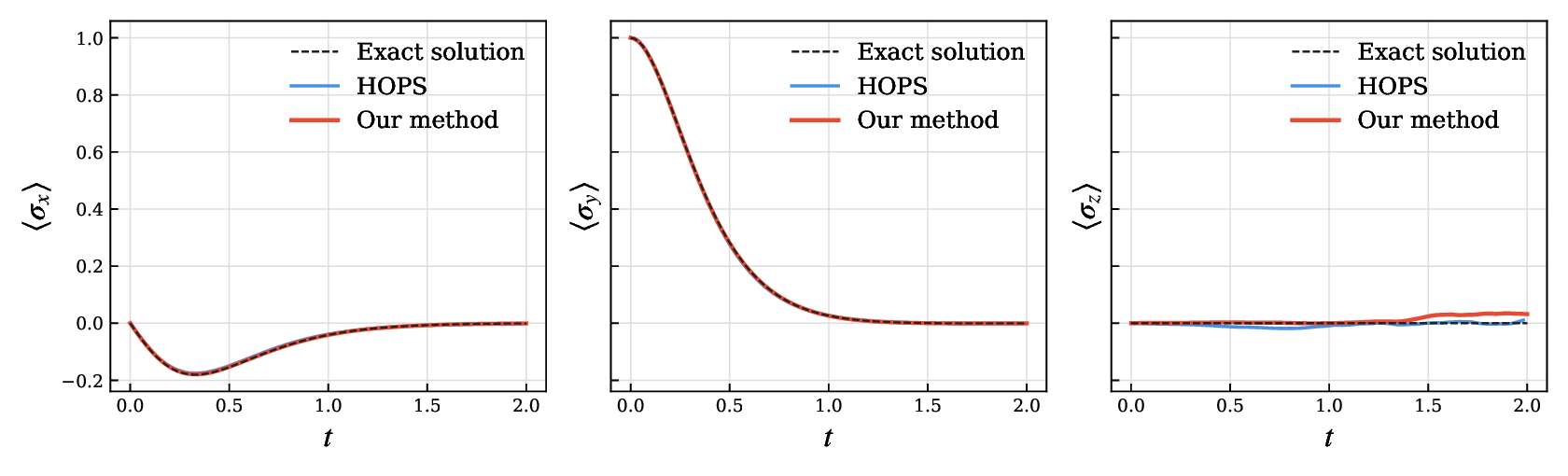}
    \caption{Evolution of $\langle\sigma_x\rangle_t$, $\langle\sigma_y\rangle_t$ and $\langle\sigma_z\rangle_t$ for BCF \eqref{eq:simple_bcf} with $\gamma=4$, comparing the proposed framework and HOPS with $N=16$ and 50,000 samples.}
    \label{fig:simple_hops_ours}
\end{figure}

\cref{fig:simple_hops_ours} compares the numerical results obtained from the proposed framework \eqref{eq:relax_hierarchy} with those from the HOPS method. The agreement between the two approaches demonstrates that for exponential or sum-of-exponential BCFs, the proposed framework recovers the HOPS results by relaxing the low-rank approximation requirements. This observation suggests that the proposed framework serves as a generalization of the HOPS method, extending its applicability to a broader class of BCFs.

\section{Conclusion}
\label{sec:conclusion}

We have developed a low-rank hierarchical framework for the non-Markovian stochastic Schrödinger equation. By approximating the bath correlation function through Mercer's theorem, the non-local stochastic dynamics are reformulated as a coupled hierarchical system that generalizes HOPS without relying on a prescribed multi-exponential expansion of the memory kernel. Under the uniqueness assumption, we proved convergence of the finite-rank, finitely truncated approximation in the $L^1$ sense as both the rank $r$ and the truncation order $N$ increase. The numerical experiments support the theoretical convergence results and illustrate the effectiveness of the proposed method.

Several issues remain for future work. These include extending the analysis to spatially dependent models with unbounded system--bath couplings, establishing well-posedness results that remove the imposed uniqueness assumption, and deriving sharper error estimates to guide the choice of $r$ and $N$ in practical computations.

\appendix
\section{Equivalence between the relaxed hierarchical framework and the HOPS for exponential BCFs}
\label{sec:appendix}
In this section, we prove the equivalence between the relaxed hierarchical framework \eqref{eq:relax_hierarchy} and HOPS for the exponential BCF $\alpha(t,s)=g e^{-w(t-s)}$ for $t\geq s$. Setting $\lambda = g$, $U(t) = e^{-wt}$, and $V(t) = e^{wt}$ in \eqref{eq:relax_hierarchy} yields
\begin{equation*}
    \frac{\partial}{\partial t}\psi_t^{(k)} = (-iH+z_tL)\psi_t^{(k)} + kg e^{wt}L\psi_t^{(k-1)} - e^{-wt}L^{\dagger}\psi_t^{(k+1)}.
\end{equation*}
Denote $\widetilde{\psi}_t^{(k)} = e^{-kwt}\psi_t^{(k)}$. It follows that
\begin{equation*}
\begin{aligned}
    \frac{\partial}{\partial t}\widetilde{\psi}_t^{(k)} =& -kw e^{-kwt}\psi_t^{(k)} + (-iH+z_tL)e^{-kwt}\psi_t^{(k)} \\
    &+ kgL e^{-(k-1)wt}\psi_t^{(k-1)} - L^{\dagger}e^{-(k+1)wt}\psi_t^{(k+1)} \\
    =& (-iH - kw + z_tL)\widetilde{\psi}_t^{(k)} + kgL\widetilde{\psi}_t^{(k-1)} - L^{\dagger}\widetilde{\psi}_t^{(k+1)},
\end{aligned}
\end{equation*}
which is identical to the hierarchical system derived via HOPS (cf. \cite[Eq. (13)]{suess2014HierarchyStochasticPure}). The same transformation applies to BCFs expressed as sums of exponentials.

\section*{Acknowledgments}
The authors would like to thank Quanhui Zhu, Hongfei Zhan, and Shuigen Liu for their valuable suggestions and discussions.

\bibliographystyle{siamplain}
\bibliography{references}
\end{document}


\maketitle

\section{Proof of the infinite series solution}
\label{sup:proof}
In this section, we verify that \eqref{eq:infinite_series_solution} is indeed a mild solution to \eqref{eq:NMSSE}, as stated in \cref{thm:analytical_solution}. 

It is clear that $\psi_t \in D(\mathcal{R}_t)$ for all $t \in [0,T]$. To show that \eqref{eq:infinite_series_solution} satisfies the mild formulation \eqref{eq:mild_formulation}, we invoke Remark \ref{rmk:strong_mild}, which implies that it suffices to demonstrate that $\widetilde{\psi}_t = S(-t)\psi_t$ satisfies the strong form \eqref{eq:strong_solution} or \eqref{eq:NMSSE_interaction} of the NMSSE in the interaction picture. We establish this by differentiating each term $\widetilde{\psi}_t^{(n,m)}$ in the series expansion. First, we rewrite $\widetilde{\psi}_t^{(n,m)}$ as an integral over the $n$-dimensional hypercube $[0,t]^n$:
\begin{equation*}
\begin{aligned}
    \widetilde{\psi}_t^{(n,m)} =& \frac{(-1)^m}{n!}\int_{[0,t]^n}\diff^{n}\boldsymbol{\tau}z_{\tau_n}\cdots z_{\tau_1} \\
    &\quad\ \int_{0<s_1<\cdots<s_{2m}<t}\diff^{2m}\boldsymbol{s} \sum_{\boldsymbol{q}\in\mathscr{Q}(\boldsymbol{s})}\alpha(q_{2m},q_{2m-1})\cdots\alpha(q_2,q_1) \\ 
    &\quad\ \mathscr{T}\left\{\widetilde{L}(\tau_n)\cdots\widetilde{L}(\tau_1)\widetilde{L}^{\dagger}(q_{2m})\widetilde{L}(q_{2m-1})\cdots\widetilde{L}^{\dagger}(q_2)\widetilde{L}(q_1)\right\}\psi_0.
\end{aligned}
\end{equation*}
Applying the chain rule, we split the time derivative of $\widetilde{\psi}_t^{(n,m)}$ into two parts:
\begin{equation*}
    \frac{\partial}{\partial t}\widetilde{\psi}_t^{(n,m)} = \mathcal{D}_{\boldsymbol{\tau}}^{(n,m)} + \mathcal{D}_{\boldsymbol{s}}^{(n,m)},
\end{equation*}
where $\mathcal{D}_{\boldsymbol{\tau}}^{(n,m)}$ and $\mathcal{D}_{\boldsymbol{s}}^{(n,m)}$ correspond to derivatives with respect to the $\tau$- and $s$-integrals, respectively. It is immediate that $\mathcal{D}_{\boldsymbol{\tau}}^{(0,m)}=0$, and for $n\geq 1$,
\begin{equation*}
\begin{aligned}
    \mathcal{D}_{\boldsymbol{\tau}}^{(n,m)} =& \frac{(-1)^m}{(n-1)!}\int_{[0,t]^{n-1}}\diff^{n-1}\boldsymbol{\tau}z_tz_{\tau_{n-1}}\cdots z_{\tau_1} \\
    &\int_{0<s_1<\cdots<s_{2m}<t}\diff^{2m}\boldsymbol{s} \sum_{\boldsymbol{q}\in\mathscr{Q}(\boldsymbol{s})}\alpha(q_{2m},q_{2m-1})\cdots\alpha(q_2,q_1) \\
    &\widetilde{L}(t)\mathscr{T}\left\{\widetilde{L}(\tau_{n-1})\cdots\widetilde{L}(\tau_1)\widetilde{L}^{\dagger}(q_{2m})\widetilde{L}(q_{2m-1})\cdots\widetilde{L}^{\dagger}(q_2)\widetilde{L}(q_1)\right\}\psi_0 \\
    =& z_t\widetilde{L}(t) \widetilde{\psi}_t^{(n-1,m)}.
\end{aligned}
\end{equation*}
For the second part, $\mathcal{D}_{\boldsymbol{s}}^{(n,0)}=0$, and for $m\geq 1$, we have
\begin{equation}
\begin{aligned}
    \mathcal{D}_{\boldsymbol{s}}^{(n,m)} =& \frac{(-1)^m}{n!}\int_{[0,t]^n}\diff^n\boldsymbol{\tau}z_{\tau_n}\cdots z_{\tau_1}\int_{0<s_1<\cdots<s_{2m-1}<t}\diff^{2m-1}\boldsymbol{s} \\
    &\sum_{k=1}^{2m-1}\alpha(t,s_k)\widetilde{L}^{\dagger}(t)
    \sum_{q\in\mathscr{Q}(s_1,\cdots,s_{k-1},s_{k+1},\cdots,s_{2m-1})}\alpha(q_{2m-2},q_{2m-3})\cdots\alpha(q_2,q_1) \\
    &\mathscr{T}\left\{\widetilde{L}(s_k)\widetilde{L}(\tau_n)\cdots\widetilde{L}(\tau_1)\widetilde{L}^{\dagger}(q_{2m-2})\widetilde{L}(q_{2m-3})\cdots\widetilde{L}^{\dagger}(q_2)\widetilde{L}(q_1)\right\}\psi_0
    \label{eq:app_D_s}
\end{aligned}
\end{equation}
To identify $\mathcal{D}_{\boldsymbol{s}}^{(n,m)}$, we state the following claim.
\begin{claim}
    For $n\geq 0, m\geq 1$,
    \begin{equation}
        \mathcal{D}_{\boldsymbol{s}}^{(n,m)} = -\widetilde{L}^{\dagger}(t)\int_{0}^t\diff s\alpha(t,s)\frac{\delta}{\delta z_s}\widetilde{\psi}_t^{(n+1,m-1)}.
    \label{eq:app_proof_1}
    \end{equation}
\end{claim}
\begin{proof}
    First, for $s\in[0,t]$,
    \begin{equation}
    \begin{aligned}
        \frac{\delta}{\delta z_s}\widetilde{\psi}_t^{(n+1,m-1)} =& \frac{(-1)^{m-1}}{(n+1)!}\sum_{k=1}^{n+1}\int_{[0,t]^{n+1}}\diff^{n+1}\boldsymbol{\tau}\delta(\tau_k-s)\prod_{\ell=1,\ell\neq k}^{n+1}z_{\tau_{\ell}} \\
        &\int_{0<s_1<\cdots<s_{2m-2}<t}\diff^{2m-2}\boldsymbol{s}\sum_{\boldsymbol{q}\in\mathscr{Q}(\boldsymbol{s})}\alpha(q_{2m-2}, q_{2m-3})\cdots\alpha(q_2,q_1) \\
        & \mathscr{T}\left\{\prod_{\ell=1}^{n+1}\widetilde{L}(\tau_{\ell})\widetilde{L}^{\dagger}(q_{2m-2})\widetilde{L}(q_{2m-3})\cdots\widetilde{L}^{\dagger}(q_2)\widetilde{L}(q_1)\right\}\psi_0 \\
        =& \frac{(-1)^{m-1}}{n!}\int_{[0,t]^n}\diff^{n}\boldsymbol{\tau}\prod_{\ell=1}^{n}z_{\tau_{\ell}} \\
        &\int_{0<s_1<\cdots<s_{2m-2}<t}\diff^{2m-2}\boldsymbol{s}\sum_{\boldsymbol{q}\in\mathscr{Q}(\boldsymbol{s})}\alpha(q_{2m-2}, q_{2m-3})\cdots\alpha(q_2,q_1) \\
        & \mathscr{T}\left\{\widetilde{L}(s)\prod_{\ell=1}^{n}\widetilde{L}(\tau_{\ell})\widetilde{L}^{\dagger}(q_{2m-2})\widetilde{L}(q_{2m-3})\cdots\widetilde{L}^{\dagger}(q_2)\widetilde{L}(q_1)\right\}\psi_0.
    \label{eq:app_proof_2}
    \end{aligned}
    \end{equation}
    Next, observe that
    \begin{align*}
        \int_{0}^t\diff s\int_{0<s_1<\cdots<s_{2m-2}<t}\diff^{2m-2}\boldsymbol{s} &= \sum_{k=1}^{2m-1}\int_{0<s_1<\cdots<s_{k-1}<s<s_{k}<\cdots<t}\diff s\diff^{2m-2}\boldsymbol{s} \\
        &= \sum_{k=1}^{2m-1}\int_{0<s_1<\cdots<s_{2m-1}<t}\diff^{2m-1}\boldsymbol{s}.
    \end{align*}
    Substituting \eqref{eq:app_proof_2} into
    \[
        \widetilde{L}^{\dagger}(t)\int_0^t\diff s\,\alpha(t,s)\frac{\delta\widetilde{\psi}_t^{(n+1,m-1)}}{\delta z_s}
    \]
    and rearranging terms yields \eqref{eq:app_D_s}, which proves \eqref{eq:app_proof_1}.
\end{proof}

We can now summarize the recursion as follows:
\begin{equation*}
\begin{cases}
    \displaystyle \frac{\partial}{\partial t}\widetilde{\psi}_t^{(n,m)} = z_t\widetilde{L}(t)\widetilde{\psi}_t^{(n-1,m)} - \widetilde{L}^{\dagger}(t)\mathcal{R}_t\widetilde{\psi}_t^{(n+1,m-1)}, & \text{for}\ n\geq1, m\geq 1, \\[2ex]
    \displaystyle \frac{\partial}{\partial t}\widetilde{\psi}_t^{(0,m)} = - \widetilde{L}^{\dagger}(t)\mathcal{R}_t\widetilde{\psi}_t^{(1,m-1)}, & \text{for}\ m\geq0, \\[2ex]
    \displaystyle \frac{\partial}{\partial t}\widetilde{\psi}_t^{(n,0)} = z_t\widetilde{L}(t)\widetilde{\psi}_t^{(n-1,0)}, & \text{for}\ n\geq0.
\end{cases}
\end{equation*}
Using the absolute convergence established in \cref{thm:analytical_solution}, we may sum the above identities over all $n,m\geq 0$. By linearity, this recovers \eqref{eq:NMSSE_interaction}, and therefore \eqref{eq:infinite_series_solution} satisfies \eqref{eq:mild_formulation}.

We next verify that \eqref{eq:infinite_series_solution} satisfies conditions \eqref{eq:mild_condition_a} and \eqref{eq:mild_condition_b}. In view of the estimate \eqref{eq:upper_bound_inf_series} and the boundedness of the operator $L$, condition \eqref{eq:mild_condition_a} is readily satisfied. To establish \eqref{eq:mild_condition_b}, we follow the derivation in \eqref{eq:app_proof_1} and adopt the procedure outlined in \eqref{eq:absolute_converg_estimate}--\eqref{eq:4_10}, which yields
\begin{equation*}
\begin{aligned}
    \left\|\mathcal{R}_t\widetilde{\psi}_t^{(n,m)}\right\|_{\mathcal{H}} \leq& \left(\|L\|_{\mathcal{H}}\int_0^t|\alpha(t,s)|\diff s\right) \frac{1}{(n-1)!}\left(\|L\|_{\mathcal{H}}\int_0^t|z_s|\diff s\right)^{n-1} \\
    &\frac{1}{m!}\left(\frac{\|L\|_{\mathcal{H}}^2}{2}\int_0^t\int_0^t|\alpha(u,v)|\diff u\diff v\right)^m  \|\psi_0\|_{\mathcal{H}}.
\end{aligned}
\end{equation*}
Summing up over $n$ and $m$ yields
\begin{equation*}
    \left\|\mathcal{R}_t\psi_t\right\|_{\mathcal{H}} \leq \left(\|L\|_{\mathcal{H}}\int_0^t|\alpha(t,s)|\diff s\right) \mathcal{M}_t \|\psi_0\|_{\mathcal{H}},
\end{equation*}
where $\mathcal{M}_t$ is defined as in \eqref{eq:upper_bound_inf_series}. We then have
\begin{equation*}
    \int_0^T \left\|\mathcal{R}_t\psi_t\right\|_{\mathcal{H}} \diff t < \infty, \quad \text{a.s.},
\end{equation*}
which confirms \eqref{eq:mild_condition_b}.
